\documentclass[a4paper,11pt,oneside]{article}
\usepackage[margin=1in]{geometry}
\usepackage{amsmath, amsthm, amssymb}
\usepackage{hyperref}
\usepackage{mathtools}
\usepackage{natbib}
\usepackage[utf8]{inputenc}

\hypersetup{
	colorlinks=true,
  linkcolor=red,          %
  citecolor=blue,         %
  filecolor=magenta,      %
  urlcolor=cyan           %
}

\usepackage[utf8]{inputenc} %
\usepackage[T1]{fontenc}    %
\usepackage{hyperref}       %
\usepackage{url}            %
\usepackage{booktabs}       %
\usepackage{amsfonts}       %
\usepackage{nicefrac}       %
\usepackage{microtype}      %

\usepackage{mathtools}
\usepackage{amsthm, amssymb}

\newtheorem{theorem}{Theorem}
\newtheorem{lemma}[theorem]{Lemma}
\newtheorem{definition}{Definition}[section]

\newtheorem{corollary}[theorem]{Corollary}
\newtheorem{remark}{Remark}
\newtheorem{claim}[theorem]{Claim}

\usepackage{amsthm}
\usepackage{algorithm}
\usepackage[noend]{algpseudocode}
\floatname{algorithm}{Procedure}

\usepackage{mathtools}
\usepackage{bbm}

\newcommand{\e}{\eps}

\usepackage[mathscr]{euscript}
\usepackage{amssymb}

\usepackage{mathtools}
\newcommand*{\myov}[1]{\overbracket[0.65pt][-1pt]{#1}}

\usepackage{xcolor}

\usepackage{stmaryrd}  %
\newcommand{\pred}[1]{\chr[#1]}
\newcommand{\nrm}[1]{\left\Vert #1 \right\Vert}

\newcommand{\Norm}[1]{\left\Vert #1 \right\Vert}
\newcommand{\norm}[1]{\Vert #1 \Vert}
\newcommand{\E}{\mathop{\mathbb{E}}}
\newcommand{\W}{\mathop{\mathbb{W}}}
\renewcommand{\P}{\mathop{\mathbb{P}}}
\newcommand{\R}{\mathbb{R}}

\newcommand{\F}{\mathscr{F}}
\newcommand{\G}{\mathcal{G}}
\renewcommand{\H}{\mathcal{H}}

\newcommand{\paren}[1]{\left( #1 \right)}
\newcommand{\sqprn}[1]{\left[ #1 \right]}

\newcommand{\set}[1]{\left\{ #1 \right\}}

\newcommand{\abs}[1]{\left| #1 \right|}

\newcommand{\del}{\partial}

\newcommand{\ddel}[2]{\frac{\del#1}{\del#2}}

\newcommand{\beq}{\begin{eqnarray*}}
\newcommand{\eeq}{\end{eqnarray*}}
\newcommand{\beqn}{\begin{eqnarray}}
\newcommand{\eeqn}{\end{eqnarray}}
\newcommand{\ben}{\begin{enumerate}}
\newcommand{\een}{\end{enumerate}}
\newcommand{\bit}{\begin{itemize}}
\newcommand{\eit}{\end{itemize}}
\providecommand{\hide}[1]{}

\newcommand{\eps}{\varepsilon}
\newcommand{\evalat}[2]{\left.#1\right|_{#2}}

\newcommand{\ceil}[1]{\ensuremath{\left\lceil#1\right\rceil}}
\newcommand{\floor}[1]{\ensuremath{\left\lfloor#1\right\rfloor}}

\newcommand{\argmin}{\mathop{\mathrm{argmin}}}

\newcommand{\inv}{^{-1}} %

\newcommand{\chr}{\boldsymbol{{1}}} %

\newcommand{\X}{\Omega}
\newcommand{\Xn}{X_{[n]}}

\newcommand{\polylog}{\operatorname{polylog}}
\newcommand{\ddim}{\operatorname{ddim}}
\newcommand{\diam}{\operatorname{diam}}
\newcommand{\Lip}{\operatorname{Lip}}

\newcommand{\lip}[1]{\nrm{#1}_{\Lip}}
\newcommand{\covn}{\mathscr{N}}

\renewcommand{\d}{\mathrm{d}}

\newcommand{\barLs}{{\myov{\Lip}}}
\newcommand{\barLw}{{\widetilde{\Lip}}}
\newcommand{\supeta}{^\eta}

\newcommand{\Lp}[1]{{\sf{L}}_{#1}}
\newcommand{\La}{\Lambda}
\newcommand{\bLa}{{\myov{\Lambda}}}
\newcommand{\bLas}{{\myov{\Lambda}}}
\newcommand{\bLaw}{{\widetilde{\Lambda}}}

\newcommand{\bLalt}{{\myov{\Lambda}}^{\sf{alt}}}
\newcommand{\joint}{\nu}
\newcommand{\serr}{\widehat{\gerr}}
\newcommand{\gerr}{\operatorname{err}}

\newcommand{\Int}[1]{{\pred{{#1}>1/2}}}
\newcommand{\supp}{\operatorname{supp}}

\newcommand{\Nmsc}[3]{\Norm{#1}_{\Lp{#2}(#3)}} %
\newcommand{\nmsc}[3]{\norm{#1}_{\Lp{#2}(#3)}} %
\newcommand{\namb}[1]{\nmsc{#1}{2}{\mu}} %
\newcommand{\ndsc}[1]{\nmsc{#1}{2}{\pi}} %
\newcommand{\nemp}[1]{\nmsc{#1}{2}{\mu_n}} %
\newcommand{\ninf}[1]{\norm{#1}_{\infty}} %
\newcommand{\extname}{\textup{\textsf{PMSE}}}
\newcommand{\ci}{c}
\newcommand{\ifrac}[2]{{#1}/{#2}}
\newcommand{\vmax}[2]{{#1}\vee{#2}}
\newcommand{\Del}[2]{\abs{f(#1)-f(#2)}}
\newcommand{\defrad}{{\eta/\ell}}
\newcommand{\radem}{\mathfrak{R}}

\def\longto{\mathop{\longrightarrow}\limits}
\newcommand{\ntoinf}{\longto_{n\to\infty}}

\renewcommand{\phi}{\varphi}

\title{Functions with average smoothness:\\ structure, algorithms, and learning}
\author{Yair Ashlagi, Lee-Ad Gottlieb, Aryeh Kontorovich \\\small{\texttt{ashlagi@post.bgu.ac.il,leead@ariel.ac.il,karyeh@cs.bgu.ac.il}}}

\begin{document}
\maketitle

%%<AK_TEX_SCRIPT_CODE_DON'T_ALTER_OR_DUPLICATE|"abs.tex"
\begin{abstract}%
  We initiate a program of average smoothness analysis for efficiently learning real-valued functions on metric spaces. Rather than using the 
  Lipschitz constant as the regularizer, we define a local slope 
  at each point and gauge the function complexity as the average of these values. Since the mean can be dramatically smaller than the maximum, this complexity measure can yield considerably sharper generalization bounds --- assuming 
  that these admit a refinement where the 
  Lipschitz
  constant is replaced by our average of local slopes.

Our first major contribution is to obtain just such distribution-sensitive bounds. This required overcoming a number of technical challenges,
  perhaps the most formidable of which was bounding the
  {\em empirical} covering numbers,
  which can
  be much worse-behaved than the ambient ones.
  Our combinatorial results are accompanied by efficient algorithms for
smoothing the labels of the random   
sample, as well as guarantees that the extension
from the sample to the whole space will
continue to be, with high probability, smooth on average.
  Along the way we discover a surprisingly rich combinatorial and
  analytic structure in the function class we define.
\end{abstract}

%%<AK_TEX_SCRIPT_CODE_DON'T_ALTER_OR_DUPLICATE|"intro.tex"
\section{Introduction}
{\em Smoothness} is a natural  measure of complexity
commonly used
in learning theory and statistics.
Perhaps the simplest method of quantifying the smoothness of a function is via the Lipschitz seminorm.
The latter has the advantage of being an analytically
and algorithmically
convenient,
broadly applicable
complexity measure,
requiring only a metric space (as opposed to additional differentiable structure).
In particular, the Lipschitz constant yields immediate bounds on the fat-shattering dimension
\citep{DBLP:journals/tit/GottliebKK14+colt},
covering numbers
\citep{MR0124720},
and sample compression
\citep{gkn-ieee18+nips}
of a function class, which in turn directly imply generalization bounds for classification and regression,
and also bounds the run-time of associated learning algorithms.

The simplicity of the Lipschitz seminorm, however,
has a downside:
it is a
worst-case measure,
insensitive to the underlying distribution.
As such, it can be overly pessimistic
in that
a single point pair can drive the Lipschitz constant of a function arbitrarily high, even if the function is nearly
constant everywhere else. Intuitively, we expect the complexity of learning a function that is highly smooth,
apart from 
low-density regions of high fluctuation,
to be determined by its average --- rather than worst-case --- behavior.
To this end, we seek a complexity measure that is resilient to
local fluctuations in low-density regions.
Formalizing this intuition
and exploring its analytic and algorithmic ramifications
is the main contribution of this paper.

Very roughly speaking, to learn
an $L$-Lipschitz (in the Euclidean metric)
function $f:[0,1]^d\to[0,1]$ at fixed precision and confidence
requires on the order of $L^d$ examples \citep{9781108498029}, and this continues to hold
in more general metric spaces \citep{DBLP:conf/nips/Kpotufe11,KpotufeDasgupta2012,DBLP:conf/nips/KpotufeG13,GottliebKK13-simbad+IEEE,DBLP:journals/tit/GottliebKK14+colt,DBLP:journals/corr/ChaudhuriD14}.
The goal of this paper is to replace the worst-case Lipschitz constant $L$ by an average one $\bar L$,
while still obtaining bounds of the genral form $\bar L^d$. Further,
we seek fully empirical generalization bounds, without making any a priori assumptions on
either the target function or the distribution.

\subsection{Our contributions}
A detailed roadmap of our results is given
in Section~\ref{sec:main-res}; here we only provide a brief overview.
We initiate a program of average-case smoothness analysis for efficiently learning
binary and
real-valued functions on metric spaces.
To any function $f:\X\to\R$ acting on a metric probability space $(\X,\rho,\mu)$,
we associate a complexity measure $\bLa_f=\bLa_f(\X,\rho,\mu)\in[0,\infty]$,
which corresponds to an average 
slope.
Our measure always satisfies
$\bLa_f\le\lip{f}$
and, as we illustrate below, the gap can be considerable.
Having defined our notion of average smoothness,
we show that the worst-case Lipschitz constant
$L$ can essentially be replaced
by its averaged variant 
$\bLa_f$ in the 
covering number
bounds.

Our results are fully empirical in that we
make no a priori assumptions on the target function or the sampling distribution,
and only require a finite diameter and doubling dimension of the metric space.
A curious and 
unique feature of our setting
--- which also presents
the bulk of the technical challenges --- is the fact that
although our hypothesis class is fixed before observing the data,
it is defined in terms of the unknown sampling distribution,
and hence not explicitly known to the learner.
This is in stark contrast with all previous 
supervised
learning settings,
where the function classes are fully known a priori.
Having observed a sufficiently large sample allows the learner
to construct an explicit hypothesis and conclude that, with high
probability, it belongs to the average smoothness class (to which our
generalization bounds then apply).

The statistical generalization bounds are accompanied by
efficient algorithms for performing sample
smoothing and a Lipschitz-type extension
for label prediction on test points.
The function classes we define turn out to exhibit
a surprisingly rich structure, making them
an object worthy of future study.
See Section~\ref{sec:main-res} for
a comprehensive overview of our techniques and central results,
along with comparisons to the current state-of-art bounds.

%%<AK_TEX_SCRIPT_CODE_DON'T_ALTER_OR_DUPLICATE|"rel-work.tex"
\subsection{Related work}
For the line segment metric $(\X,\rho)=([a,b],\abs{\cdot})$,
the {\em bounded variation} (BV)
of any $f:[a,b]\to\R$ with integrable derivative is given by
$V_a^b(f)=\int_a^b\abs{f'(x)}\d x$;
this is perhaps the most basic notion of average smoothness.
BV
does not require differentiability; see
\citet{MR3156940} for an encyclopedic reference.
Generalization bounds for BV functions may be obtained via covering numbers
\citep{DBLP:journals/tit/BartlettKP97,DBLP:journals/iandc/Long04}
(the latter also gave
an efficient algorithm for learning BV functions
via linear programming)
or the fat-shattering dimension
\citep[Theorem 11.12]{MR1741038}.
The aforementioned results correspond to the
case of
a uniformly distributed $\mu$ on $[a,b]$,
and thus are not
distribution-sensitive.
A natural extension of BV to
general measures would be to define $V_a^b(f)=\int_a^b\abs{f'(x)}\d\mu(x)$, but then the known fat-shattering
and covering number estimates break down --- especially if $\mu$ is not known to the learner.

Generalizing the notion of BV to higher dimensions is not nearly as straightforward. A common approach
is via the Hardy-Krause variation \citep{MR3156940,MR0419394,MR2208697}.
Even the two-dimensional case evades a simple characterization; counter-intuitively, Lipschitz functions $f:[0,1]^2\to\R$
may fail to have finite variation in the Hardy-Krause sense \citep[Lemma 1]{MR3514716}.
Some (rather loose) $\Lp1$ covering numbers for BV functions on $[0,1]^n$ were obtained by \citet[Theorem 3.1]{MR3852571};
these 
are not distribution-sensitive.
Generalizations of BV to metric
measure spaces beyond the Euclidean are known \citep{MR3642646}; we are not aware of any covering number
or combinatorial dimension
estimates for these.

If one considers {\em bracketing} (rather than covering) numbers, there are known results for controlling these in terms
of various measures of average smoothness. \citet{MR2324525} bound the bracketing numbers of Besov- and Sobolev-type classes.
\citet{MR2761605} also gave bracketing number bounds,
using a different notion of smoothness: the
{\em averaged modulus of continuity} developed by \citet{MR995672}.
We note that covering numbers asymptotically always give tighter
estimates than bracketing ones \citep{steve-cover-bracket}.
More significantly, to our knowledge, all of the previous
results bound the {\em ambient}
rather than the {\em empirical} covering numbers (see Section~\ref{sec:def} for definitions,
and Section~\ref{sec:cov-num} for bounds),
when it is precisely the latter that are needed for Uniform
Glivenko-Cantelli laws \citep{MR757767}.

A seminal work on recovering functions with
spacially inhomogeneous smoothness from noisy
samples is \citet{MR1635414}.
More in the spirit of our program is
the notion of {\em Probabilistic Lipschitzness}
\citep{urner2013probabilistic},
which
seeks to relax a hard Lipschitz condition
on the 
labeling 
function in binary classification. The authors position
it as a ``data niceness'' condition, analogous to that in \citet{mammen1999}.
These significantly differ from
our notion of average slope.
Most importantly, PL and the various Tsybakov-type noise conditions
are {\em assumptions} on the data-generating distribution
rather then {\em empirically computable} quantities on a given sample.
Our approach is fully empirical in the sense of not making
a priori assumptions on the distribution or the target function. Additionally,
PL is specifically designed for binary classification with
deterministic labels --- unlike our notion, which is applicable to
any real-valued function and any conditional
label distribution.

In this paper, we make systematic use of a Lipschitz-type extension (\extname, defined in Section~\ref{sec:PMSE})
explicitly tailored to our framework.
This extension is closely related to one introduced by
\citet{MR2431047} (also relevant is \citet{MR3649229},
p. 385 and p. 416, Remark 5.1).

%%<AK_TEX_SCRIPT_CODE_DON'T_ALTER_OR_DUPLICATE|"defs.tex"
\subsection{Definitions and preliminaries}
\label{sec:def}
\paragraph{Metric probability spaces.}
We assume a basic familiarity with metric measure spaces and refer the reader to a standard reference, such as \citet{MR1800917}.
Standard set-theoretic notation is used throughout; in particular, for $f:\X\to\R$ and $A\subset\X$, we denote
the restriction of $f$ to $A$ by $\evalat{f}{A}$.
The triple $(\X,\rho,\mu)$ is a {\em metric probability space} if $\mu$ is a 
probability measure supported on the Borel
$\sigma$-algebra
induced by the open sets of $\rho$.
For $\X$-valued random variables, the notation $X\sim\mu$
means that $\P(X\in A)=\mu(A)$ for all Borel sets $A$.

\paragraph{Covers, packings, nets, hierarchies, partitions.}
The {\em diameter} of $A\subseteq\X$ is the maximal interpoint distance: $\diam(A)=\sup_{x,x'\in A}\rho(x,x')$.
For $t>0$ and $A,B\subseteq\X$, we say that $A$ is a $t$-{\em cover} of $B$ if
\beq
\sup_{b\in B}\inf_{a\in A}\rho(a,b)\le t,
\eeq
and define the $t$-covering number of $B$
to be the minimum cardinality of any $t$-cover,
denoted by $\covn(t,B,\rho)$.
We say that $A\subseteq B\subseteq\X$ is a $t$-{\em packing} of $B$
if $\rho(a,a')>t$ for all distinct $a,a'\in A$.
Finally, $A$ is a $t$-{\em net} of $B$ if it is simultaneously a
$t$-cover and a $t$-packing. 
A 
family of sets
$H_{t^0} \subseteq H_{t^{-1}} \subseteq \ldots \subseteq H_{t^{-m}}$
is a {\em hierarchy} for the set 
$H_{t^{-m}}$ if each $H_{t^{-i}}$ ($i<m$)
is a $t^{-i}$-net of $H_{t^{-(i+1)}}$,
where we have assumed that $H_{t^{-m}}$ have diameter $1$ and so $H_1$ contains a single point.

We denote by $B(x,r)=\set{x'\in\X:\rho(x,x')\le r}$
the (closed) $r$-ball about $x$. If
there is a $D<\infty$ such that
every $r$-ball in $\X$
is contained in the union of some $D$ $r/2$-balls,
the metric space $(\X,\rho)$ is said to be {\em doubling}.
Its {\em doubling dimension} is defined as $\ddim(\X)=\ddim(\X,\rho)=:\log_2D^*$,
where $D^*$ is the smallest $D$ verifying the doubling property.
It is well-known \citep{KL04,GottliebKK13tcs+alt} that
\beqn
\label{eq:ddim-pack}
\covn(t,\X,\rho)\le\paren{\frac{2\diam(\X)}{t}}^{\ddim(\X)},
\qquad t>0
,
\eeqn
which will be referred to as the covering property of doubling spaces. 
The packing property of doubling spaces 
asserts an analogous packing number bound,
up to constants in the exponent. A hierarchy for any $n$-point $\X$ set can be constructed in time
$2^{O(\ddim(\X))} \min \{\log n, \log \Delta\}$, where
$\Delta$ is the {\em aspect ratio} (minimal
interpoint distance)
of $\X$ \citep{KL04, HM06, DBLP:conf/stoc/ColeG06}.

To any finite $V\subseteq\X$ we associate the map
$\phi_V:\X\to V$
taking each $x\in\X$ to its nearest neighbor in $V$, with ties broken
arbitrarily (say, via some fixed ordering on $\X$)\footnote{
A measurable total order always exists \citep{hksw20}.
}.
The collection of sets $\set{ \phi_V\inv(v) : v\in V}$ is said to
comprise the {\em Voronoi} partition of $\X$ induced by $V$.
If $V$ happens to be a $t$-net of $\X$, then
\beqn
\label{eq:part-diam}
\rho(x,\phi_V(x))\le t,
\qquad x\in\X
.
\eeqn

\paragraph{Indices, norms.}
We write $[n]:=\set{1,\ldots,n}$ and use the shorthand $z_{[n]}:=(z_1,\ldots,z_n)$
for sequences.
For any metric probability space $(\X,\rho,\mu)$,
$p\ge1$,
and any $f:\X\to\R$,
we define the norm
$\norm{f}_{
\Lp{p}(\X,\rho,\mu)
}^p=\E_{X\sim\mu}[|f(X)|^p]
=\int_\X |f(x)|^p\d\mu(x)
$.

This work assumes a single fixed
metric probability space
$(\X,\rho,\mu)$;
this will be termed the {\bf ambient space}.
Several derived
metric probability spaces
$(\X',\rho',\mu')$ will be considered,
which will all be {\bf induced subspaces}
of $(\X,\rho)$ in the sense that
$\X'\subseteq\X$ and 
$\rho'=\evalat{\rho}{\X'\times\X'}$.
To lighten the notation, we will
often suppress the common metric $\rho$ 
and
use the shorthand
$
\nmsc{\cdot}{p}{\mu'}
:=
\norm{\cdot}_{\Lp{p}(\X',\rho,\mu')}
$.
For any $f,g:\X\to\R$ and any induced subspace $(\X',\rho)$ of $\X$ with measure $\mu'$,
  we use the shorthand
  \beqn
  \label{eq:evalat}
  \nmsc{f-g}{p}{\mu'} := \Nmsc{
    \evalat{f}{\X'}
    -
    \evalat{g}{\X'}
  }{p}{\mu'}
=
\norm{
    \evalat{f}{\X'}
    -
    \evalat{g}{\X'}
}_{\Lp{p}(\X',\rho,\mu')}
  .
  \eeqn
In particular, sampling $\X_n:=\Xn=(X_1,\ldots,X_n)\sim\mu^n$
induces the {\bf empirical space}
$(\X_n,\rho,\mu_n)$
with the norm $\nmsc{\cdot}{p}{\mu_n}$,
where $\mu_n$
is the empirical measure on $\X_n$, formally given by
$\mu_n(x)=n\inv\sum_{i=1}^n\pred{X_i=x}$.
The $\ell_\infty$ norm $\norm{f}_\infty=\sup_{x\in\X}\abs{f(x)}$
is measure-independent
and dominates all of the 
measure-induced
norms:
\beq
\nmsc{f-g}{p}{\mu'}
\le
\ninf{f-g},\qquad p\ge 1.
\eeq
The Lipschitz seminorm $\lip{f}$
is the smallest $L\in[0,\infty]$
satisfying $\abs{f(x)-f(x')}\le L\rho(x,x')$ for all $x,x'\in\X$.

\paragraph{Strong and weak mean.}
We define the {\bf weak mean}
of a
non-negative random variable $Z$ by
\beqn
\label{eq:weakmean}
\W[Z]
=
\sup_{t>0}t\P(Z\ge t).
\eeqn
In contrast, the {\bf strong mean} is just the usual expectation $\E[Z]$.
By
Markov's inequality,
we always have
$\W[Z]\le\E[Z]$;
further, the latter might be infinite while the former is finite.
A partial reverse inequality for finite measure spaces is given in
Lemma~\ref{lem:strong-weak-log}.

\paragraph{Local and average %
slope.}
For $f:\X\to\R$,
we define the 
{\em
slope} of $f$ at $x\in\X$
with respect to an $A\subseteq\X$
as
\beqn
\label{eq:locdef}
\La_f(x,A)
:= \sup_{x'\in A\setminus\set{x}}\frac{|f(x)-f(x')|}{\rho(x,x')}
.
\eeqn
Thus, 
\beqn
\label{eq:lipLa}
\lip{f}=\sup_{x\in\X}\La_f(x,\X).
\eeqn
We will define two notions of {\em average} slope:
strong and weak, corresponding, respectively,
to the strong and weak $\Lp1$ norms
of the random variable $\La_f(X)$, where $X\sim\mu$.
The two averages are defined, respectively, as
\beqn
\label{eq:blas}
\bLas_f(\mu,\X)
&:=& \E_{X\sim\mu}
\sqprn{\La_f(X,\X)}
=\nmsc{\La_f(\cdot,\X)}{1}{\mu},\\
\label{eq:blaw}
\bLaw_f(\mu,\X)
&:=& \W_{X\sim\mu}
\sqprn{\La_f(X,\X)}
=
\sup_{t>0}t\mu(M_f(t))
,
\eeqn
where $M_f(t)$ is the $t$-{\em level set}, a central object in this paper:
\beqn
\label{eq:Mf}
M_f(t):=\set{x\in\X:\La_f(x,\X)\ge t}.
\eeqn
The strong-weak mean inequality above implies that
\beqn
\label{eq:blasw}
\bLaw_f(\mu,\X)
\le
\bLas_f(\mu,\X)
\le
\lip{f}
\eeqn
always holds (the second inequality is obvious); further, 
$\bLas_f(\mu,\X)$
might be infinite 
while 
$\bLaw_f(\mu,\X)$
is finite
(as demonstrated by the step function on $[0,1]$ with the uniform measure, see Section~\ref{sec:disc}).
Since the above definitions were stated for
{\em any} metric probability space,
$\La_f(x,\X_n)$,
$\bLas_f(\mu_n,\X_n)$,
and
$\bLaw_f(\mu_n,\X_n)$
are well-defined as well.
(To appreciate the subtle choice of our definitions,
note that some intuitively appealing variants
irreparably fail, as discussed in Section~\ref{sec:disc}.)

The collection of all
  $[0,1]$-valued
      $L$-Lipschitz functions on $\X$,
  as well as
  its strong and weak mean-slope counterparts
  are denoted, respectively, by
  \beqn
  \label{eq:lipdef}
  \Lip_L(\X,\rho) &=& \set{f\in[0,1]^\X;~ \lip{f}\le L}, \\
  \label{eq:slipdef}
  \barLs_L(\X,\rho,\mu) &=& \set{f\in[0,1]^\X;~ \bLas_f(\mu,\X)\le L},\\
\label{eq:wlipdef}
  \barLw_L(\X,\rho,\mu) &=& \set{f\in[0,1]^\X;~ \bLaw_f(\mu,\X)\le L}.
  \eeqn
It follows from (\ref{eq:blasw}) that
$
\Lip_L(\X,\rho,\mu)\subset
\barLs_L(\X,\rho,\mu)\subset
\barLw_L(\X,\rho,\mu)
$,
where all containments are, in general, strict,
and
\beqn
\label{eq:markov}
\mu(M_f(L/t))\le t,
\qquad t>0
\eeqn
holds for all
$f\in\barLw_L(\X,\rho,\mu)$.
For most of this paper, we shall be interested in the larger latter
class, but occasional results for $\barLs_L(\X,\rho,\mu)$ 
will be presented, when of
independent interest.

{\em Remark:} Observe that
the classes $\barLs$ and $\barLw$
are defined in terms of the unknown sampling
distribution $\mu$. Given full knowledge of a function
$f:\X\to\R$, a learner can verify that $f\in\Lip_L$
but, absent full knowledge of $\mu$, it is impossible to know
for certain whether $f\in\barLw_L$ (or $f\in\barLs_L$).
As increasingly larger samples are observed,
the learner will be able to assert the latter
inclusions with increasing confidence.

\paragraph{Empirical and true risk.}
For any probability measure $\joint$ on $\X\times[0,1]$, we associate to any measurable
$f:\X\to\R$ its
{\em risk} $R(f;\joint):=\E_{(X,Y)\sim\joint}|f(X)-Y|$.
In the special case of the empirical measure $\joint_n$ induced by
a sample $(X_i,Y_i)_{i\in[n]}\sim\joint^n$, $R(f;\joint_n)$
is the empirical risk. For regression with real-valued $f$, this is the $\Lp1$-risk;
for classification with $\set{0,1}$-valued $f$, this is the $0$-$1$ error. (See \citet{mohri-book2012} for a standard reference.)

\paragraph{Miscellanea.} Additional standard inequalities and notations are deferred to Section~\ref{sec:misc} in the Appendix.

\pagebreak

%%<AK_TEX_SCRIPT_CODE_DON'T_ALTER_OR_DUPLICATE|"main-res.tex"
\section{Main results and roadmap}
\label{sec:main-res}

This section assumes a familiarity with the terminology
and notation defined in Section~\ref{sec:def}.

\paragraph{Combinatorial structure.}
Our point of departure is Theorem~\ref{thm:cov-num}, which bounds
the $\Lp2(\mu)$ 
({\em ambient})
covering numbers of the function class
$\barLw_L(\X,\rho,\mu)$ --- 
and, a fortiori, of 
$\barLs_L(\X,\rho,\mu)$ 
[defined in (\ref{eq:slipdef}, \ref{eq:wlipdef})]
---
in terms of the average slope $L$, $\diam(\X)$, and $\ddim(\X)$.
Crucially, there is no dependence on the 
Lipschitz
constant $\lip{\cdot}$. 
At scale $t$, Theorem~\ref{thm:cov-num} gives a bound of roughly
\beqn
\label{eq:rough}
\left( {L}/{t} \right)^{\tilde{O}(\ddim)}
\eeqn
instead of the previous state-of-the art bound of
$( {\lip{\cdot}}/{t} )^{\tilde{O}(\ddim)}$.
The improvement can be dramatic, as the worst-case may be significantly (even infinitely) larger
than the mean (\ref{eq:lipLa}).

This simple result appears to be novel and interesting
in its own right,
but is insufficient to guarantee generalization bounds
(via a Uniform Glivenko-Cantelli law), since the latter
require control over the {\em empirical} (i.e., $\Lp2(\mu_n)$)
covering numbers.
Bounding these 
proved to be a formidable challenge.
The calculation in Theorem~\ref{thm:distances}
reduces 
this problem
to the one of bounding
the empirical measure of the level set, $\mu_n(M_f(\ell))$,
uniformly over all the functions in our class.
We make the perhaps surprising
discovery that
(i) uniform control over the $\mu_n(M_f(\ell))$ is possible
for the sub-class of functions free of certain local defects
(Lemma~\ref{lem:M_f(ell)})
and (ii) any $f\in \barLw_L(\X,\rho,\mu)$ 
is approximable in $\ell_\infty$ by a defect-free function
(Lemma~\ref{lem:repair});
see the beginning of Section~\ref{sec:defects}
for some discussion and intuition.
Together, these enable us to overcome the central
challenge of controlling the empirical
covering numbers (Theorem~\ref{thm:emp-cov-num}), yielding a bound comparable to (\ref{eq:rough}). The implied generalization bounds (Section~\ref{sec:gen})
enjoy a dependence on $\sqrt{L}/n^{1/8d}$, while all previously
known generalization results for classification and regression
feature a dependence on $\lip{\cdot}$ \citep{MR2013911,9781108498029,GottliebKK13-simbad+IEEE}.

\paragraph{Optimization and learning.}
From the perspective of supervised learning theory,  
our statistical bounds imply a 
non-trivial
algorithmic problem:
Given a labeled sample,
produce a hypothesis
whose true risk does not significantly exceed its empirical risk
(with high probability).
These notions are briefly defined in Section~\ref{sec:def}
and discussed in more detail in Section~\ref{sec:risk-bounds}.
In light of the aforementioned generalization bounds,
the learning procedure may be recast as follows:
The learner is given a ``complexity budget'' $L>0$.
Given a labeled sample, 
$(X_i,Y_i)\in \X\times[0,1]$, $i\in[n]$,
the learner
seeks to fit to the data some function 
with average slope not exceeding $L$,
while minimizing the empirical risk.
The latter is induced by
either the  $0$-$1$ loss (classification)
or the $\Lp1$ loss (regression).
Approximation algorithms for this problem
are presented in Section~\ref{sec:learn}.
Briefly, we cast the regression problem as an
optimization problem amenable to the mixed 
packing-covering framework of \citet{KY-14},
and further improve the algorithmic run-time by 
reducing the number of constraints in the program
(Section~\ref{sec:learn-reg}).
Interestingly, the classification problem
admits an efficient bi-criteria approximation
when casting the ``smoothness budget'' in terms of
the weak mean, but we were unable to find
an efficient solution for the strong mean, and provide
some indication that this may in fact be a hard problem
(Section~\ref{sec:learn-class}).

\paragraph{Adversarial extension.}
Having solved the learning problem, we have obtained an
approximate minimizer of the empirical risk,
but this does not immediately imply a bound on the true risk.
To obtain such a bound,
we 
demonstrate that with high probability, average smoothness under
the empirical measure $\mu_n$
translates to average smoothness
under the true sampling measure $\mu$,
from which a bound on true risk follows.

To this end, we define the following 
{\em adversarial extension} problem:
An adversary draws $n$ points $\X_n\subset \X$ from $\mu$ and labels them with $y\in[0,1]$.
This induces an average slope
$\bLas_y$ (or $\bLaw_y$)
under the empirical measure $\mu_n$. The adversary's goal
is to force {\em any} extension of $y$ from $\X_n$ to all of $\X$ to have
a significantly larger average slope under the true measure $\mu$.
In the case of regression, we show that if the learner is willing to tolerate a small 
distortion of $y$ under $\Lp1(\mu_n)$, it is possible to guarantee an at-most constant 
factor
increase in
both
$\bLas$ and $\bLaw$
with high probability (Section~\ref{sec:adv-ext}).
In the case of classification, we show how to achieve an most $2^{O(\ddim)}\polylog(n)$ factor increase (with high probability),
without incurring any distortion (Section~\ref{sec:adv-ext-class}).

%%<AK_TEX_SCRIPT_CODE_DON'T_ALTER_OR_DUPLICATE|"cov-num.tex"

\section{Covering numbers}
\label{sec:cov-num}

Our covering-numbers results will be stated for 
$\barLw_L$.
It follows from
(\ref{eq:blasw})
that these results hold verbatim 
for $\barLs_L$ as well.
(These function classes are defined in
(\ref{eq:lipdef}, \ref{eq:slipdef}, \ref{eq:wlipdef})).

\subsection{Ambient covering numbers}

Our empirical covering-numbers results build upon the following simpler result for the ambient covering numbers:

\begin{theorem}[Ambient $\Lp2$ Covering Numbers]
  \label{thm:cov-num}
  For $L,t>0$,
\beq
\log
\covn(t,
\barLw_L(\X,\rho,\mu),
\Lp2(\mu)
) &\le&
\covn(t^{3}/(64L),\X,\rho)
\log(16/t)
.
\eeq
In particular, for doubling spaces with $\diam(\X)\le1$, we have
\beq
\log
\covn(t,
\barLw_L(\X,\rho,\mu),
\Lp2(\mu)
)
&\le&
\paren{\frac{128L}{t^3}}^{\ddim(\X)}\log\frac{16}t.
\eeq
\end{theorem}

The proof of Theorem \ref{thm:cov-num}
will be based upon
following result:
\begin{lemma}[\citet{GottliebKK13-simbad+IEEE}, Lemma 5.2]
  \label{lem:lipcov}
For $L,t>0$,
\beq
\log
\covn(t,
\Lip_L(\X,\rho),
\ell_\infty) &\le&
\covn(t/(8L),\X,\rho)
\log(8/t)
.
\eeq
In particular, for doubling spaces with $\diam(\X)=1$, we have
\beq
\log\covn(t,
\Lip_L(\X,\rho),
\ell_\infty) &\le&
\paren{\frac{16L}{t}}^{\ddim(\X)}
\log\frac8t
.
\eeq
\end{lemma}

\begin{proof}[Proof of Theorem~\ref{thm:cov-num}]
Recall the definition of 
the level set $M_f(\cdot)$ in (\ref{eq:Mf}).
By
(\ref{eq:markov}),
we have
$\mu(M_f(L/t))\le t$
for any
$f\in\barLw_L(\X,\rho,\mu)$
and $t>0$,
and by construction, 
$\La_f(x,\X)\le L/t$
for all $x\in
M'_f(L/t)
:=
\X\setminus M_f(L/t)
$.
Thus, for all $f\in\barLw_L(\X,\rho,\mu)$, we have
\beqn
\label{eq:Mft}
\evalat{f}{M'_f(L/t)} \in \Lip_{L/t}(M'_f(L/t),\rho).
\eeqn
Let $\hat F$ be a $t/2$-cover of
$\Lip_{4L/t^2}(\X,\rho)$
under $\ell_\infty$.
We claim that $\hat F$ is a $t$-cover of $\barLw_L(\X,\rho,\mu)$ under $\Lp2(\X,\rho,\mu)$.
Indeed, choose an $f\in\barLw_L(\X,\rho,\mu)$. It follows from (\ref{eq:Mft})
that
\beqn
\evalat{f}{M'_f(4L/t^2)} \in \Lip_{4L/t^2}(M'_f(4L/t^2),\rho).
\eeqn
Via the McShane-Whitney Lipschitz extension \citep{MR1562984,Whitney1934},
there is an $\tilde f\in\Lip_{4L/t^2}(\X,\rho)$ coinciding with $f$ on
$M'_f(4L/t^2)$.
Since $\hat F$ is a $t/2$-cover, there is an $\hat f\in \hat F$
such that $\namb{\tilde f-\hat f}\le\ninf{\tilde f-\hat f}\le t/2$.
Therefore
\beq
\namb{f-\hat f} &\le&
\namb{f-\tilde f} + \namb{\tilde f-\hat f} \\
&=&
\paren{\int_\X(f(x)-\tilde f(x))^2\d\mu(x)}^{1/2} + \namb{\tilde f-\hat f} \\
&=&
\paren{\int_{
    M_f(4L/t^2)    
  }(f(x)-\tilde f(x))^2\d\mu(x)}^{1/2} + \namb{\tilde f-\hat f} \\
&\le&
\paren{\int_{
    M_f(4L/t^2)
  }1^2\d\mu(x)}^{1/2} + \namb{\tilde f-\hat f} \\
&\le& t/2+t/2=t.
\eeq
The claim follows from 
Lemma~\ref{lem:lipcov}, which bounds the size of (a minimal) $\hat F$.
\end{proof}

\subsection{Empirical covering numbers}
The main result of this section is a bound on the empirical covering numbers. To avoid trivialities,
we state our asymptotic bounds in $n$
under the assumption
that $\ddim(\X),L\ge1$.

\begin{theorem}[Empirical $\Lp2$ Covering Numbers]
  \label{thm:emp-cov-num}
  Let $(\X,\rho,\mu)$ be a doubling metric measure space
  (the ambient space)
  with $\diam(\X)\le1$ and $(\X_n,\rho,\mu_n)$ its empirical realization.
  
Then, for constant $\delta>0$ and $L,\ddim(\X)\ge1$, we have that
\beq
\log
\covn(
(\alpha+1)\eps_0
,
\barLw_L(\X,\rho,\mu),
\Lp2(\mu_n) 
)
&\le&
\paren{\frac{L}{\alpha^3\eps_0^3}}^{\ddim(\X)}
\log\frac{1}{\alpha\eps_0},
\qquad \alpha>0
\eeq
holds with probability at least $1-2\delta$,
where
$\eps_0\le C_\delta \sqrt{L}n^{-\ifrac{1}{8d}}$
and $C_\delta>0$ is a universal constant.
\end{theorem}
The proof will be given below, and follows directly
from
Theorem~\ref{thm:cov-num}
and the following result:
\begin{theorem}[Preserving distances
between $\Lp2(\mu)$ and $\Lp2(\mu_n)$]
  \label{thm:distances}
  Let $(\X,\rho,\mu)$ be a doubling metric measure space
  (the ambient space)
  with $\diam(\X)\le1$ and $(\X_n,\rho,\mu_n)$ its empirical realization.
  Then, with probability at least $1-2\delta$, we have that all
  $f,g\in \barLw_L(\X,\rho,\mu)$
  satisfy  
  \beq
  \nemp{f-g} &\le& 6\namb{f-g} 
  \\  &+& 
  25n^{-\ifrac{1}{4d}} +
  15\sqrt{L}n^{-\ifrac{1}{8d}} +
  (6+2^{d/4})n^{-\ifrac{1}{8}} +
  \left(
  \frac{162}{n}\log\frac2\delta
  \right)
  ^{\ifrac{1}{4}},
  \eeq
  where $d=\ddim(\X)$ and
  we adhere to the notational convention in (\ref{eq:evalat}).
\end{theorem}

\begin{proof}
Let $r,t,\eta>0$ be parameters to be chosen later.
Our first step is to approximate the function class 
$\barLw_L(\X,\rho,\mu)$ by its ``$\eta$-smoothed'' version
\beq
{\barLw_L}\supeta(\X,\rho,\mu)
:=\set{f^\eta:f\in \barLw_L(\X,\rho,\mu)}
\subseteq \barLw_L(\X,\rho,\mu)
,
\eeq
where $f^\eta$ is the function constructed in Lemma~\ref{lem:repair},
when the latter is invoked with the parameter $\ell=L/t$.
In particular, $\ninf{f-f^\eta}\le4\eta$ and
$\La_{f^\eta}(x,\X)\le\La_{f}(x,\X)$ for all $x\in\X$.
Thus, for $f,g\in\barLw_L(\X,\rho,\mu)$,
\beqn
\nonumber
\nemp{f-g} &\le& \ninf{f-f^\eta}+\ninf{g-g^\eta}+\nemp{f^\eta-g^\eta}\\
&\le& 8\eta+\nemp{f^\eta-g^\eta},
\eeqn
and so it will suffice to bound the latter.

Let $V\subset\X$ be an $r$-net of $(\X,\rho)$;
by the doubling property (\ref{eq:ddim-pack}),
\beqn
\label{eq:vsize}
|V|\le(2/r)^{\ddim(\X)}.
\eeqn
The net $V$ induces the Voronoi partition $\X=\bigcup_{v\in V}W(v)$,
such that for each cell $W(v)$ we have
$x \in W(v) \implies \rho(x,v) \leq r$, as well as
the measure under $\mu$, denoted by $\pi(v):=\mu(W(v))$.
Together, these
induce the 
finite
metric measure space
$
(V,\rho,\pi)
$.
The map
$\phi_V:\X\to V$
takes
each $x\in\X$ with its Voronoi cell;
thus,
$\phi_V\inv(\phi_V(x))=W(v)$ for all $x\in W(v)$.

The proof proceeds in several steps,
in which always $f,g\in{\barLw_L}\supeta(\X,\rho,\mu)$
and the notational convention in (\ref{eq:evalat}) is used.
Define $
M_f(L/t),
M_g(L/t)
\subset\X$
as in (\ref{eq:Mf}).
As in the proof of Theorem~\ref{thm:cov-num},
we have that $f$ is $L/t$-Lipschitz on
$M'_f(L/t):=\X\setminus
M_f(L/t)
$, with extension $\tilde f\in\Lip_{L/t}(\X,\rho)$.
Define $\tilde g\in\Lip_{L/t}(\X,\rho)$ analogously,
as extending $
\evalat{g}{M'_g(L/t)}
\in\Lip_{L/t}(M'_g(L/t),\rho)$.

\paragraph{Comparing
the norms
  $\nemp{f-g}^2
=\frac1n\sum_{i=1}^n(f(X_i)-g(X_i))^2
  $
  and
  $\ndsc{\tilde f-\tilde g}^2
  =
\sum_{v\in V}(\tilde f(v)-\tilde g(v))^2\pi(v)  
  $.}
We begin by invoking (\ref{eq:abc}):
\beqn
\label{eq:nemp<}
\nemp{f-g}^2
\le
3\nemp{\tilde f-\tilde g}^2
+
3\nemp{f-\tilde f}^2
+
3\nemp{g-\tilde g}^2.
\eeqn
The second and third terms
in the bound (\ref{eq:nemp<})
are bounded identically:
\beqn
\nemp{f-\tilde f}^2
&=&
\frac1n\sum_{i=1}^n(f(X_i)-\tilde f(X_i))^2\nonumber\\
&\le&
\frac1n\sum_{i:X_i\in M_f(L/t)}1
\;=\;
\label{eq:muM}
\mu_n(M_f(L/t))
.
\eeqn

To estimate the first term in the bound (\ref{eq:nemp<}),
recall that
$\tilde f$ is $L/t$-Lipschitz on $\X$
and $\rho(x,\phi_V(x))\le
r$
(and the same holds for $\tilde g$),
whence
\beq
|\tilde f(X_i)-\tilde g(X_i)|
&\le&
|\tilde f(X_i)-\tilde f(\phi_V(X_i))|
+
|\tilde g(X_i)-\tilde g(\phi_V(X_i))|
+
|\tilde f(\phi_V(X_i))-\tilde g(\phi_V(X_i))|\\
&\le&
|\tilde f(\phi_V(X_i))-\tilde g(\phi_V(X_i))|
+
2Lr/t.
\eeq
Using $(a+b)^2\le2a^2+2b^2$,
this yields
\beq
\nemp{\tilde f-\tilde g}^2
&=&
\frac1n\sum_{i=1}^n(\tilde f(X_i)-\tilde g(X_i))^2\\
&\le&
\frac2n\sum_{i=1}^n(\tilde f(\phi_V(X_i))-\tilde g(\phi_V(X_i)))^2
+8(Lr/t)^2\\
&=&
2\sum_{v\in V}(\tilde f(v)-\tilde g(v))^2\mu_n(W(v))
+8(Lr/t)^2\\
&=&
2
\nmsc{\tilde f-\tilde g}{2}{\pi_n}^2
+8(Lr/t)^2,
\eeq
where $\pi_n$ is the measure on $V$ given by $\pi_n(v)=
\mu_n(W(v))$.
Observe that
\beq
\abs{
\ndsc{\tilde f-\tilde g}^2
-
\nmsc{\tilde f-\tilde g}{2}{\pi_n}^2
}
&\le&
\sum_{v\in V}(\tilde f(v)-\tilde g(v))^2|\pi(v)-\pi_n(v)|\\
&\le&
\sum_{v\in V}|\pi(v)-\pi_n(v)|=\norm{\pi-\pi_n}_1.
\eeq

A bound on $\norm{\pi-\pi_n}_1
$
is provided by (\ref{eq:TV-bk}):
with probability at least $1-\delta$,
\beq
\norm{\pi-\pi_n}_1 
\le 
\sqrt{\frac{|V|}{n}}+\sqrt{\frac2n\log\frac2\delta}
\le
\sqrt{\frac{
(2/r)^{\ddim(\X)}    
  }{n}}+\sqrt{\frac2n\log\frac2\delta}
  .
\eeq
To bound
$\mu_n(M_f(L/t))$ in (\ref{eq:muM}),
we invoke Corollary~\ref{cor:supMf}:
with probability at least $1-\delta$,
\beq
\sup\set{
  \mu_n(M_f(L/t))
  :
  f\in{\barLw_L}\supeta(\X,\rho,\mu)
  }
\le
24t +
\frac12\sqrt{\frac{
(2L/\eta t)^{\ddim(\X)}
  }{n}}+
\frac12\sqrt{\frac2n\log\frac2\delta}
.
\eeq
These calculations culminate in the bound
\beqn
\label{eq:nemp<ndsc}
\nemp{f-g}^2 &\le& 6\ndsc{\tilde f-\tilde g}^2 +
24(Lr/t)^2 + 144t 
\\
&+&
\nonumber
6\sqrt{\frac{(2/r)^{\ddim(\X)}}{n}} +
9\sqrt{\frac2n\log\frac2\delta} +
3\sqrt{\frac{(2L/\eta t)^{\ddim(\X)}}{n}}
.
\eeqn

\paragraph{Comparing
the norms
  $\ndsc{\tilde f-\tilde g}^2
  $
and
  $\namb{f-g}^2$
  .}

Since $\tilde f$ is $L/t$-Lipschitz and $\diam(W(v))\le r$, we have
$\tilde f(W(v))\subseteq f(v)\pm Lr/t$, and analogously for $\tilde g$.
It follows that
\beqn
\label{eq:abs-ub}
|\tilde f(v)-\tilde g(v)|
&\le&
|\tilde f(x)-\tilde g(x)|
+2Lr/t
,
\qquad x\in W(v)
\eeqn
and hence

\beqn
\label{eq:sqr-ub}
(\tilde f(v)-\tilde g(v))^2
&\le&
2(\tilde f(x)-\tilde g(x))^2
+8(Lr/t)^2
,
\qquad x\in W(v)
.
\eeqn
Integrating,
\beq
\ndsc{\tilde f-\tilde g}^2=
\sum_{v\in V}(\tilde f(v)-\tilde g(v))^2\pi(v)
&=&
\sum_{v\in V}\int_{W(v)}(\tilde f(v)-\tilde g(v))^2\d\mu(x)\\
&\le&
\sum_{v\in V}\int_{W(v)}[
2(\tilde f(x)-\tilde g(x))^2
+8(Lr/t)^2  
]\d\mu(x)\\
&=&
2\namb{\tilde f-\tilde g}^2
+8(Lr/t)^2.
\eeq

Using
(\ref{eq:abc}) and
the triangle inequality,
we have
\beq
\namb{\tilde f-\tilde g}^2
&\le&
3\namb{f-g}^2
+
3\namb{f-\tilde f}^2
+
3\namb{g-\tilde g}^2
\\
&\le&
3\namb{f-g}^2
+
6t
\eeq
(since
each of
$\namb{f-\tilde f},\namb{g-\tilde g}$
is at most $\sqrt t$).

Combining these yields
\beqn
\label{eq:ndsc<namb}
\ndsc{\tilde f-\tilde g}^2 &\le&
6\namb{f-g}^2+12t+8(Lr/t)^2.
\eeqn

\paragraph{Finishing up.}
Combining (\ref{eq:nemp<ndsc}) and (\ref{eq:ndsc<namb}) yields
\beqn
\label{eq:nemp<namb}
\nemp{f-g}^2 &\le& 36\namb{f- g}^2+
72(Lr/t)^2 + 216t \\
\nonumber
&+&
6\sqrt{\frac{(2/r)^{\ddim(\X)}}{n}} +
9\sqrt{\frac2n\log\frac2\delta} +
3\sqrt{\frac{(2L/\eta t)^{\ddim(\X)}}{n}}
.
\eeqn
\paragraph{Choosing $r,t,\eta$.} 
Putting $d=\ddim(\X)$, we choose
$r = n^{-\ifrac{1}{2d}}$, 
$t = Ln^{-\ifrac{1}{4d}}$,
and
$\eta = 2n^{-\ifrac{1}{4d}}$.
For this choice,
\beq
\nemp{f-g}^2 &\le& 36\namb{f- g}^2+
72n^{-{1}/{2d}} + 216Ln^{-{1}/{4d}} \\
&+&
2^{d/2}\cdot6n^{-{1}/{4}} +
9\sqrt{\frac2n\log\frac2\delta} +
3\left(
n^{1/4d}
\right)^{d/2}n^{-\ifrac{3}{8}}.
\eeq

Applying the inequality $\sqrt{a+b}\le\sqrt{a}+\sqrt{b}$ to (\ref{eq:nemp<namb}) 
proves the claim.
\end{proof}

\begin{proof}[Proof of Theorem~\ref{thm:emp-cov-num}]
Let $\eps_+$ be the additive term
in the bound
in Theorem~\ref{thm:distances}:
\beqn
\label{eq:eps+}
\eps_+
=
  25n^{-\ifrac{1}{4d}} +
  15\sqrt{L}n^{-\ifrac{1}{8d}} +
  (6+2^{d/4})n^{-\ifrac{1}{8}} +
  \left(
  \frac{162}{n}\log\frac2\delta
  \right)
  ^{\ifrac{1}{4}}.
\eeqn
Since for fixed $\delta>0$
and $d,L\ge1$
the dominant term in (\ref{eq:eps0})
is 
$15\sqrt{L}n^{-\ifrac{1}{8d}}$,
and so
there is a $C=C_\delta<0$
such that
\beqn
\label{eq:eps0}
\eps_+\le C_\delta \sqrt{L}n^{-\ifrac{1}{8d}}=:\eps_0.
\eeqn
Now Theorem~\ref{thm:distances} implies that
any $\eps$-cover of 
$\barLw_L(\X,\rho,\mu)$
under $\Lp2(\mu)$
also provides a
$(6\eps+\eps_0)$-cover
under $\Lp2(\mu_n)$.
Equivalently, an 
$\alpha\eps_0/6$-cover
of the former
yields an
$(\alpha+1)\eps_0$-cover
of the latter, for $\alpha>0$.
Hence,
\beq
\log
\covn(
(\alpha+1)\eps_0
,
\barLw_L(\X,\rho,\mu),
\Lp2(\mu_n) 
)
&\le&
\log
\covn(
\alpha\eps_0/6
,
\barLw_L(\X,\rho,\mu),
\Lp2(\mu)
)
\\
&\le&
\paren{\frac{cL}{\alpha^3\eps_0^3}}^{\ddim(\X)}
\log\frac{1}{\alpha\eps_0},
\eeq
where $c>0$ is a universal constant.

\end{proof}

%%<AK_TEX_SCRIPT_CODE_DON'T_ALTER_OR_DUPLICATE|"defects.tex"
\section{Defect free functions}
\label{sec:defects}

This section presents results that were invoked in the proofs in Section \ref{sec:cov-num}.
It constitutes the core of the analytic and combinatorial
structure we discovered in the
very general setting
of real-valued functions on metric spaces.
Such a function may fail to be on-average smooth for two ``moral''
reasons: due to ``large jumps'' or ``small jumps''. The former is
witnessed by two nearby points $x,x'$ for which $|f(x)-f(x')|$ is large
--- say, $1$. The latter is witnessed by two nearby (say, $\eps$-close)
points $x,x'$ for which $\eps\ll|f(x)-f(x')|\le T(\eps)\ll1$ --- say, $T(\eps)=\sqrt\eps$.
It turns out that the large jumps do not present a problem
for the combinatorial structure we seek in
Lemma~\ref{lem:M_f(ell)}, which forms the basis for
Corollary~\ref{cor:supMf}, the latter a crucial component in
the empirical covering number bound, Theorem~\ref{thm:distances}.
Rather, it is the small jumps --- which we formalize as {\em defects} below ---
that present an obstruction. Fortunately, as we show in
Lemma~\ref{lem:repair}, {\em any} bounded real-valued function on a 
doubling
metric space admits a defect-free approximation under $\ell_\infty$.

\subsection{Definition and structure}
\label{sec:defects-defs}
For a given $f:\X\to\R$,
$\ell>0$,
and $x,y\in\X$,
we say that
$y$ is
an $\ell$-{\bf slope witness} for $x$
(w.r.t. $f$)
if ${\Del{x}{y}}/{\rho(x,y)}\ge\ell$.
For
$\eta,\ell>0$
and $\ci\ge1$,
we say that
an $x\in\X$ is an $(\eta,\ell,\ci)$-{\bf defect} of $f$ if:
\begin{enumerate}
    \item[(a)] $\La_f(x)\ge\ell$
    \item[(b)]
Every $\ell/\ci$-slope witness $y$ of $x$ verifies $\Del{x}{y} \le \eta$.
\end{enumerate}
Define $\Xi_f(\eta,\ell,\ci)\subseteq\X$ to be the set of all
$(\eta,\ell,\ci)$-defects of $f$.
Note that
$
\Xi_f(\eta,\ell,\ci)
\subseteq
\Xi_f(\eta,\ell,\ci')
$
whenever $\ci'\le\ci$.
For $\eta,\ell>0$, $\ci \geq 1$ define $\G(\eta,\ell,\ci)$ to be the collection of all
$f:\X\to[0,1]$ such that $f$ does not have any $(\eta,\ell,\ci)$-defects.

\begin{lemma}[Combinatorial structure of defect-free functions]
  \label{lem:M_f(ell)}
For every $\eta,\ell>0$, $\ci \geq 1$
there is a partition $\Pi=\set{B_1,\ldots,B_{N}}$ of $\X$
  of size $N\le (2\ell/\eta)^{\ddim(\X)}$
  such that for each
 $f\in\G(\eta,\ell,\ci)$,
we have
\beqn
\label{eq:MUM}
M_f(\ell) \subseteq U_f \subseteq M_f(\ell/4c),
\eeqn
where $M_f(\cdot)$ is the level set defined in (\ref{eq:Mf})
and
\beqn
\label{eq:Uf}
U_f := \bigcup\set{B\in\Pi: B\cap M_f\neq\emptyset}.
\eeqn
\end{lemma}
\begin{proof}
  Let $\Pi$ be the Voronoi partition induced by
  an $\eta/\ell$-net of $\X$.
  Then the claimed bound on $|\Pi|$ holds by (\ref{eq:ddim-pack}) and the first inclusion in (\ref{eq:MUM})
  is obvious by construction; it only remains to show that $U_f \subseteq M_f(\ell/4c)$.

  Choose any $u\in U_f$. Since $\Pi$ is a net, there is some
  $x \in M_f(\ell)$ for which $\rho(x,u) \leq {\eta}/{\ell}$.
  Since $f$ has no $(\eta,\ell,\ci)$-defects and $\La_f(x)\ge\ell$, there must be some
${\ell}/{c}$-slope witness $y\in\X$ of $x$
  for which $\Del{x}{y} > \eta$.
  Invoking (\ref{eq:abc1/2}), we have
$\vmax{\Del{u}{x}}{\Del{u}{y}} \geq \Del{x}{y}/2 > \eta/2$. We consider the two cases:
\begin{enumerate}
    \item[(i)] $\rho(x,y) \leq {\eta}/{\ell}$
    \item[(ii)] $\rho(x,y) > {\eta}/{\ell}$.
\end{enumerate}
In the first case, the triangle inequality implies that 
$\vmax{ \rho(u,x)}{ \rho(u,y) } \leq {2\eta}/{\ell}$,
and hence
\beq
\La_f(u)
\geq
\vmax{ \frac{\Del{u}{x}}{\rho(u,x)}}{\frac{\Del{u}{y}}{\rho(u,y)} }
\geq
\frac{\eta / 2}{2\eta / \ell}
=
\frac{\ell}{4}
\geq
\frac{\ell}{4\ci}
\implies
u \in M_f(\ell/4\ci).
\eeq
For the second case, if $\Del{u}{x} \geq \eta/2$, the proof is the same as in the first case.
Otherwise, $\Del{u}{y} \geq \Del{x}{y}/2$ and so
\beq
\frac{\Del{u}{y}}{\rho(u,y)}
\geq
\frac{\Del{x}{y}}{2 \rho(u,y)}
\geq
\frac{\Del{x}{y}}{2 \cdot 2 \rho(x,y)}
\geq
\frac{\ell}{4\ci}
\implies
u \in M_f(\ell/4\ci),
\eeq
where the second inequality is a result of applying the triangle inequality to the fact that $\rho(u,x) < \rho(x,y)$.
\end{proof}

\subsection{Repairing defects}
The main result of this section is that the problematic ``small jumps'' alluded to in the beginning of Section~\ref{sec:defects}
can be smoothed out via an $\ell_\infty$ approximation.

\begin{lemma}[Defect repair]
\label{lem:repair}
  For each $\eta,\ell>0$ and $f:\X \to [0,1]$,
  there
is an $\bar f\in\G(\eta,\ell, \ci = 6)$ such that $\ninf{f-\bar{f}} \leq 4\eta$
and
\beqn
\label{eq:barff}
\La_{\bar f}(x,\X) \le \La_f(x,\X),\qquad x\in\X.
\eeqn
\end{lemma}
\begin{proof}
We will prove the equivalent claim
that
$\bar f\in\G(\eta/2,\ell, \ci = 6)$
and
$\ninf{f-\bar{f}} \leq 2\eta$.
  We begin by constructing $\bar f$.
  Let $M_f(\ell)$ be as defined in (\ref{eq:Mf}) and $V$ be a $\defrad$-net of this set.
  Partition $V=V_0\cup V_1$, where $V_0$ is ``smooth,''
  \beq
  V_0 := \set{v\in V: B(v,\defrad)\nsubseteq\Xi_f(\eta,\ell,1)},
  \eeq
  and $V_1$ is ``rough,''
  \beq
  V_1 := \set{v\in V: B(v,\defrad)\subseteq\Xi_f(\eta,\ell,1)}.
  \eeq

Define
\beq
A_f := \bigcup_{v_1\in V_1} B(v_1,\defrad) \setminus \paren{ \bigcup_{v_0\in V_0} B(v_0,\defrad)\cup V}.
\eeq
In words, $A_f$ consists of the entirely defective (or ``rough'') balls without their center-points or their intersections with smooth balls.
Define $\bar{f}$ as the \extname\ extension of
$f$
from
$\X \setminus A_f$
to $\X$, as in Definition~\ref{def:PMSE}.
Having constructed the $\bar f$, we proceed to verify its properties.

\paragraph{Proof that (\ref{eq:barff}) holds.}
This is an immediate consequence of Theorem~\ref{thm:PMSE}.

\paragraph{Proof that $\ninf{f-\bar{f}} \leq 2\eta$.}

Since \extname\ is an extension, 
we need only establish
$\abs{f(x)-\bar{f}(x)} \leq 2\eta$
for $x \in A_f$.
For any such $x$,
the definition of $A_f$ implies the existence of some $v_1 \in V_1$ for which $\rho(x,v_1) \leq \eta / \ell$.
Since $\abs{f(x)-\bar{f}(x)} \leq \abs{f(v_1)-f(x)} +  \abs{\bar{f}(x)-f(v_1)}$, it is sufficient to bound each term separately by $\eta$.

To bound the first term, assume, for a contradiction, that $\Del{v_1}{x} > \eta$.
Then $\frac{
\Del{v_1}{x}
}{\rho(v_1,x)} > \frac{\eta}{\eta/\ell} = \ell$, contradicting the defectiveness of $v_1$.

To bound the second term, again assume for a contradiction
that
$\bar{f}(x)-f(v_1) > \eta$;
the case $
f(v_1)
-
\bar{f}(x)
> \eta$
is handled analogously.
Our assumption
implies
$\La_{\bar{f}}(x) \geq \frac{\bar{f}(x)-\bar{f}(v_1)}{\rho(x,v_1)} \geq\ell$.
By properties (i) and (v) of Corollary~\ref{cor:PMSE-prop},
there is an
$x' \in \X\setminus A_f$ for which $\La_{\bar{f}}(x) = \frac{\bar{f}(x')-\bar{f}(x)}{\rho(x',x)} \geq \ell$ and $\bar{f}(x') \geq \bar{f}(x)$.
Invoking (\ref{eq:fxyz}),
we have $\frac{f(x')-f(v_1)}{\rho(x',v_1)} \geq \ell$.
Additionally, we have $\bar{f}(x') - \bar{f}(v_1) > \bar{f}(x) - \bar{f}(v_1) > \eta$, again contradicting the defectiveness of $v_1$.

\paragraph{Proof that $\bar f\in\G(\eta / 2,\ell, 6)$.}
  A statement equivalent to $\bar f\in\G(\eta / 2,\ell, 6)$ is that
$\Xi_{\bar f}(\eta/2,\ell,6)=\emptyset$.
Let us define the sets
\beq
E_1 &:=& \bigcup_{v_0\in V_0} B(v,\defrad),\\
E_2 &:=& A_f \cup V_1,\\
E_3 &:=& \X\setminus M_f(\ell),
\eeq 
which are, by construction, a
(not necessarily disjoint) cover of $\X$. Hence, it suffices to show that $\Xi_{\bar f}(\eta/2,\ell,6)\cap E_i=\emptyset$ for $i\in[3]$.

Let $x\in E_3$. By Theorem~\ref{thm:PMSE},
$\La_{\bar{f}}(x,\X)
\leq \La_f(x,\X)$.
Since $x \notin M_f(\ell)$, we have that $\La_{\bar{f}}(x) < \ell$, implying that
$x$ is not $(\eta/2,\ell,\ci)$-{defect}
for any $\ci\ge1$ with respect to $\bar{f}$.

Let $x \in E_1$. Then there is a $v_0 \in V_0$ such that $\rho(x,v_0) \leq \eta / \ell$.
Being in $V_0$
implies that $v_0$ has some $\ell$-slope witness $v_0' \in E_1$
such that $\abs{f(v_0)-f(v_0')} > \eta$.
This implies by (\ref{eq:abc1/2}) that 
$\vmax{\abs{\bar{f}(x) - \bar{f}(v_0)}}
{\abs{\bar{f}(x)-\bar{f}(v_0')}} >
\abs{f(v_0)-f(v_0')} / 2 > \eta / 2$,
since
$f$ and $\bar f$ must agree on
$v_0$ and $v_0'$.
The triangle inequality yields a slope of at least $\ell/4$
witnessed by at least one of $\set{v_0,v_0'}$.

Let $x \in E_2$. Then there is a $v_1 \in V_1 \subseteq \Xi_f(\eta,\ell,1)$
for which $\rho(x,v_1) \leq \eta / \ell$. By  
Remark~\ref{rem:PMSE-argmax}, the maximal slope at $x$ 
is achieved at the two distinct points $u^*$ and $v^*$ by which it is determined.
Suppose $\vmax{\rho(x,u^*)}{\rho(x,v^*)} >
\eta/2\ell$. If $\La_f(x)\ge\ell$ and
one of $\set{u^*,v^*}$ --- say, $u^*$ ---
satisfies
the inequality then:
$$\abs{f(x)-f(u^*)} \geq \ell  \rho(x,u^*) >
\ell
\cdot\frac{\eta}{2\ell} = \frac{\eta}{2}.$$
This contradicts the second condition for an $(\eta/2,\ell,c)$-{defect}.
Otherwise, $\vmax{ \rho(x,u^*)}{ \rho(x,v^*) } 
\leq \eta/2\ell$.
Since $V$ is an $\eta/\ell$-net, it is not possible that both $u^*,v^* \in V$.
Therefore
(without loss of generality)
$u^* \in E_1
\cup
E_3$.
If $u^* \in E_3$,
it follows from property (ii) in Corollary~\ref{cor:PMSE-prop} 
that 
$$\La_{\bar{f}}(x,\X)
\leq
\La_{\bar{f}}(u^*,\X)
\leq
\La_f(u^*,\X)
< \ell,$$
which implies that $x \notin \Xi_{\bar{f}}(\eta/2,\ell,\ci)$ for any $\ci\ge1$.
If $u^* \in E_1$ then there is some $v_0 \in V_0$
for which $\rho(u^*,v_0) \leq \eta/\ell$.
Since $\La_{f}(v_0) > \ell$ and $v_0 \notin \Xi_f(\eta,\ell,1)$,
it must have some witness $v_0'$ such that $\abs{f(v_0)-f(v_0')} > \eta$
and the slope between them is at least $\ell$.
Similarly to previous arguments, 
$\vmax{\abs{\bar{f}(x) - \bar{f}(v_0)}}{ \abs{\bar{f}(x)-\bar{f}(v_0')} } 
>\abs{f(v_0)-f(v_0')} / 2 \geq \eta / 2$.
In either case, applying the triangle inequality yields a slope of at least $\ell/6$
witnessed by 
at least one of $\set{v_0,v_0'}$,
which shows that $x \notin  \Xi_{\bar{f}}(\eta/2,\ell,6)$.

\end{proof}

The culmination of this section is the following crucial uniform convergence result invoked in the course of proving Theorem~\ref{thm:distances}:
\begin{corollary}
  \label{cor:supMf}
  Let $f^\eta$ be the function constructed from $f$ as in Lemma~\ref{lem:repair},
when the latter is invoked with the parameter $\ell=L/t$,
and let
\beq
{\barLw_L}\supeta(\X,\rho,\mu)
:=\set{f^\eta:f\in \barLw_L(\X,\rho,\mu)}
\subseteq \barLw_L(\X,\rho,\mu)
.
\eeq
  Then,
  with probability at least $1-\delta$, we have
  \beq
\sup\set{
  \mu_n(M_f(L/t))
  :
  f\in
  {\barLw_L}\supeta(\X,\rho,\mu)
  }
&\le&
24t +
\frac12\sqrt{\frac{
(2L/\eta t)^{\ddim(\X)}
  }{n}}+
\frac12\sqrt{\frac2n\log\frac2\delta}
,
  \eeq
  where $\mu_n$ is the empirical measure
  induced by $\mu$.
\end{corollary}    
\begin{proof}
Let $\Pi$ be as in Lemma~\ref{lem:M_f(ell)}.
For each $f\in
{\barLw_L}\supeta(\X,\rho,\mu)
$,
let
$U_f$
be as defined in (\ref{eq:Uf}).
Then, invoking the inclusion in (\ref{eq:MUM}) and recalling that $\mu(M_f(L/t))\le t$ for all $f\in \barLw_L(\X,\rho,\mu)$ and
$t>0$ (and that Lemma~\ref{lem:repair} sets $c=6$),
\beq
\sup_f
  \mu_n(M_f(L/t))
  &\le& \sup_f\mu_n(U_f) \\
  &=& \sup_f\sqprn{\mu_n(U_f)-\mu(U_f)}+\mu(U_f)\\
  &\le& \sup_f\paren{\mu_n(U_f)-\mu(U_f)}+\sup_f \mu(U_f)\\
  &\le& \sup_{U\subseteq\Pi}\paren{\mu_n(U)-\mu(U)} +\sup_f \mu(M_f(L/4ct))\\
  &\le& \sup_{U\subseteq\Pi}\paren{\mu_n(U)-\mu(U)} +24t\\
  &=&
\frac12\sum_{B\in\Pi}|\mu(B)-\mu_n(B)|+24t,
\eeq
where the last step used the variational characterization of the total variation distance.
The latter is bounded as in (\ref{eq:TV-bk}),
completing the proof.

\end{proof}

%%<AK_TEX_SCRIPT_CODE_DON'T_ALTER_OR_DUPLICATE|"learning.tex"
\section{Learning algorithms: training}
\label{sec:learn}

We consider two learning problems --- classification and regression ---
in a unified 
{\em agnostic}
setting \citep{mohri-book2012}. 
In each case, the learner 
receives a labeled sample,
$(X_i,Y_i)_{i\in[n]}$,
where $X_i\in\X$ and $Y_i\in\set{0,1}$ for classification
or 
$Y_i\in[0,1]$ for regression.
The learner then selects a hypothesis
$f\in \barLw_{L}(\X,\rho,\mu)$,
where $L$ is fixed a priori.\footnote{
Assuming $L$ fixed and known incurs no loss of generality, as discussed
at the beginning of Section~\ref{sec:gen}.
} 
Finally, given a test point $x'\in\X$,
the learner's predicted label is either $f(x')\in[0,1]$
(regression) or $\Int{f(x')}\in\set{0,1}$ (classification);
this is elaborated in greater detail in Section~\ref{sec:risk-bounds}.
Computational considerations, as well as the learner's inherent uncertainty
regarding whether 
$f\in \barLw_{L}(\X,\rho,\mu)$
(see below),
will lead us to consider relaxed versions of the learning problem,
where the ``complexity budget'' will increase
from $L$
to $O(L)$ for regression
and $O(L \cdot \polylog n)$ for classification.

As described in Section~\ref{sec:def},
the sample is drawn from the joint measure $\joint$
over $\X\times[0,1]$ --- whose first marginal,
by definition, necessarily coincides with $\mu$ ---
and, once drawn, induces
the empirical measure $\joint_n$.
The {\em empirical} (respectively, {\em true}) risk
of $f:\X\to\R$ is the expected value of $|f(X)-Y|$
under $\joint$ (respectively, $\joint_n$);
these are denoted by
$R(f;\joint)$
and
$R(f;\joint_n)$.
The learner seeks to minimize 
$R(f;\joint)$ but can only directly access
$R(f;\joint_n)$;
hence, an {\em optimization} algorithm
will seek to minimize the latter, while
a {\em generalization} bound will provide
a high-confidence bound on the former.

Our learning problem presents a novel challenge,
not typically encountered in the classic 
supervised learning
setting.
Namely, ensuring that the learner's hypothesis belongs to
$\barLw_{L}(\X,\rho,\mu)$
(or
$\barLs_{L}(\X,\rho,\mu)$)
is non-trivial,
and is certainly not guaranteed ``by construction''.
Indeed, let us break down the learning process into
its basic stages.
The training stage, which may be called 
{\em smoothing}
or {\em denoising} (or yet {\em regularization}), 
involves
solving the following optimization problem:
Choose a hypothesis $f$ that stays within
the ``smoothness budget'' and achieves a low
$R(f;\joint_n)$. Algorithmically, this is done
by computing an $\hat f:\X_n\to[0,1]$, where
$\X_n=(X_i)_{i\in[n]}$ and $\hat f(X_i)$ is a ``smoothed''
version of the $Y_i$, achieving a desired average empirical slope
$\bLaw_{\hat f}(\mu_n,\X_n)$
(or  $\bLas_{\hat f}(\mu_n,\X_n)$).
The function $\hat f$ is then extended via (a variant of)
\extname\ from $\X_n$ to all of $\X$.
The novel challenge is to ensure that
$\bLaw_{\hat f}(\mu,\X)$
(respectively, $\bLas_{\hat f}(\mu,\X)$)
does not much exceed its empirical version.
We term this problem {\em adversarial extension}
and address it in Sections~\ref{sec:adv-ext} and \ref{sec:adv-ext-class}.

The results for regression are conceptually simpler
and are presented first; those for classification follow. Throughout this section, we assume
$\diam(\X)\le1$ and $d:=\ddim(\X)<\infty$.

%%<AK_TEX_SCRIPT_CODE_DON'T_ALTER_OR_DUPLICATE|"smooth-alg.tex"
\newcommand{\Lpr}{L}
\newcommand{\tlf}{\hat{f}}
\newcommand{\xn}{x_{[n]}}

\subsection{Regression}
\label{sec:learn-reg}

\begin{theorem}[Training and generalization 
for
regression, strong mean.]
\label{thm:strong-mean-learn}
Let
$\joint$ be some distribution on $\X\times[0,1]$,
and
$S_n=(X_i,Y_i)_{i\in[n]}\sim\joint^n$ be a set sampled i.i.d. from $\joint$. 
Denote
by
$\hat f \in \barLs_{\Lpr}(\X_n,\rho,\mu_n)$ the minimizer of $
R(\cdot;\joint_n)
$.
Then there is an efficient learning algorithm $\mathcal{A}$
that constructs a hypothesis $f = \mathcal{A}(S_n)$ such that such that for any 
given $L>0$, $0<\e,\delta<1$ and $c<1$:
\ben
\item[(a)]
With probability at least
$1-\exp\left( {-n (\e/8) ^ {d+1} + d\ln{(8/\e)}} \right) - 3\delta$,
\beq
R(f;\joint) 
\leq 
(1+c)R(\hat f;\joint_n) 
+
O \left(
\e L
+
\frac{C_\delta \sqrt{L}}
{n^{1/8d}}
\right)
+
\frac{C_\delta^{-d/2} \sqrt{2}} {n^{5/16}}
+
3\sqrt{\frac{\log(2/\delta)}{2n}}
,
\eeq
\item[(b)]
$f(x)$ can be evaluated at each $x \in \X$ in time $O(n^2)$ 
after a one-time ``smoothing'' computation of 
$\min \{ 2^{O(d)} (n/c^2) \log \Delta,O((n/c)^2 \log n)\}$ where $\Delta=\min_{x\neq x'\in\xn}\rho(x,x')$,
\een
where 
$C_\delta$ is a constant depending only on $\delta$. %
\end{theorem}

\begin{proof}
The smoothing algorithm described in
Lemma~\ref{lem:smth-reg} constructs
an approximate minimizer 
$\tilde f\in \barLs_{O(1)L}(\X_n,\rho,\mu_n)$
of 
$R(\cdot;\joint_n)$
and the 
``adversarial extension''
algorithm in Lemma~\ref{lem:adv-ext-strong}
provides an extension
$f$
of $\tilde f$
from $\X_n$ to $\X$
that,
with high probability, belongs to $\barLs_{O(1)L}(\X,\rho,\mu)$
and
increases 
the empirical risk by at most an additive $O(\e L)$.
The bound in (a) is then a direct
application of (\ref{eq:regbounds}).

\end{proof}

\begin{lemma}[Smoothing for regression, strong mean]
\label{lem:smth-reg}
Let $(\X,\rho)$ be a metric 
space 
with $\diam(\X)\le 1$ and $\ddim(\X)<\infty$.
Suppose that $(x,y)=(\xn,y_{[n]})
\in(\X^n,\R^n)$
and $L>0$
are given, and denote
\beqn
\label{eq:dist-reg}
A(f;x,y,\Lpr) 
&:=&
\{
\norm{f-y}_{\Lp{1}(\mu_n)} 
:
\: f \in \barLs_{\Lpr}(\xn,\rho,\mu_n) 
\},
\eeqn
where $\mu_n$ is the counting measure on $\xn$.

Then a $(1+c)$-approximate minimizer 
$\tlf \in \barLs_{\Lpr}(\xn,\rho,\mu_n)$
of
$A(\cdot;x,y,\Lpr)$
can be computed in time
\beq
\min \{ 2^{O(\ddim(\X))} (n/c^2) \log \Delta~,~O((n/c)^2 \log n)\}.
\eeq
\end{lemma}

\begin{proof}
We cast the optimization problem as a linear program
over the variables $L_i, w_i, z_i$:

\beq
    \begin{array}{lll}
    \textrm{Minimize}   & W = \sum_{i\in[n]} w_i  	&    \\
    \textrm{subject to} & \sum_{i\in[n]} L_i \le \Lpr  	&    \\
			& w_i \ge |z_i-y_i|		& \forall i \in [n]	\\
			& |z_i-z_j|    \le  L_i  \rho(x_i,x_j) & \forall i,j \in [n] \\
                        & 0 \le w_i,z_i \le 1    	& \forall i \in [n]    .
    \end{array}
\eeq

A linear program in $O(n)$ variables 
and constraints
can be solved in time 
$\tilde{O}(n^\omega)$  \citep{CLS-19},
where $\omega$ is the best exponent for matrix inversion,  
currently $\omega \approx 2.37$.

\paragraph{First runtime improvement.}
To improve on the runtime, we will utilize the packing-covering
framework of \citet{KY-14}. 
For a constraint matrix of at most $m$ rows and columns with 
all non-negative entries and at most $\zeta$ non-zero entries, 
the algorithm computes in time
$O((m/c^2) \log \zeta + \zeta)$
a $(1+c)$-approximate solution satisfying all constraints.
A difficulty in utilizing this framework is that our constraint
matrix has negative entries; in particular, each constraint
of the form
\beq
|z_i-z_j|    &\le&  L_i  \cdot \rho(x_i,x_j)
\eeq
reduces to solving two constraints of the form
\beq
z_i-z_j    &\le&  L_i  \cdot \rho(x_i,x_j) \\
z_j-z_i    &\le&  L_i  \cdot \rho(x_i,x_j).
\eeq
To address this, we introduce dummy variables $\tilde{z}_i$
satisfying $z_i + \tilde{z}_i = 1$. 
Then the above constraints become:
\beq
L_i \cdot  \rho(x_i,x_j) + \tilde{z}_i + z_j &\ge& 1 \\
L_i \cdot \rho(x_i,x_j) + z_i + \tilde{z}_j &\ge& 1 .
\eeq
Similarly, the constraint
\beq
w_i \ge |z_i-y_i|
\eeq
is replaced by two constraints
\beq
w_i+z_i &\ge& y_i \\
w_i+\tilde{z}_i &\ge& 1 - y_i
.
\eeq

For the runtime, we have that both terms
$m,\zeta$ are bounded by $O(n^2)$,
for a total runtime of $O(n^2 \log n)$.

\paragraph{Second runtime improvement.}
The main obstacle to improving the above runtime lies in the 
quadratic number of constraints necessary to compute the average
slope. Here we show that we can reduce these to only
$2^{O(\ddim)} n \log \Delta$ constraints, 
each with a constant number of variables, 
and so the linear program of \citet{KY-14} will run in time 
$2^{O(\ddim)} \tilde{O}(n \log \Delta)$.
However, this comes at a cost of increasing the average slope by a constant factor.

We first extract from $\xn=(x_1,\ldots,x_n)$ a point hierarchy 
$\{H_{2^{-k}}\}_{k=0}^{\lceil \log{\Delta}\rceil}$.
Let $P(x,k)$ be the nearest neighbor of 
$x \in \xn$ in level $H_{2^{-k}}$, and for each point
$x' \in \xn$,
let neighborhood 
$N(x',k)$ include all points $x$ for which
$P(x,k) = x'$.
(Of course, $N(x',q)$ can be non-empty only if $x' \in H_{2^{-q}}$.)
Now let representative set $C(x,k)$ include all net points in $H_k$ 
satisfying $2 \cdot 2^k \le \rho(x,y) < 4 \cdot 2^k$.

Instead of computing the mean slope averaged over all points,
we will record for each
hierarchical point $x_j \in H_k$ the maximum and minimum labels
of points in its neighborhood
($z'_{j,k},z''_{j,k}$, respectively), 
and for each point $x_i \in S$, compare its label to the 
maximum and minimum among the neighborhoods of the points of 
representative set $C(x,k)$ for all $k$.
For any point pair $x,x' \in S$, 
the triangle inequality implies that for level $k$ satisfying 
$2 \cdot 2^k \le \rho(x,x') < 4 \cdot 2^k$
we have
$2^k \le \rho(x,C(x',k))< 5 \cdot 2^k$,
and so the average slope
is preserved up to constant factors.

\beq
    \begin{array}{lll}
    \textrm{Minimize}   & W = \sum_{i\in[n]} w_i  		&    \\
    \textrm{subject to} & \frac{1}{n}\sum_{i\in[n]} L_i \le \Lpr  		&    \\
			& w_i \ge |z_i-y_i|		& \forall i \in [n]	\\
			& \max \{ |z_i- z'_{j,k}|, |z_i- z''_{j,k}|\}   \le  L_i \cdot \rho(x_i,x_j) 
								& \forall \in [n],k \in [\lceil \log{\Delta}\rceil], x_j \in C(x_i,k)	 \\
			& z''_{i,k} \le z_j \le z'_{i,k}	& \forall i \in [n],k \in [\lceil \log{\Delta}\rceil], x_j \in N(x_i,k)	\\
                        & 0 \le z_i, z'_i, z''_i \le 1		& \forall i \in [n]    .
    \end{array}
\eeq

By the packing property (\ref{eq:ddim-pack}),
each point of $S$ can be found in at most
$2^{O(\ddim)}$ neighborhoods of each level, 
so that the sum of sizes all all neighborhoods is 
$2^{O(\ddim)} n \log \Delta$.
Similarly, 
$|C(x)| = 2^{O(\ddim)} \log \Delta$,
and so the sum of sizes of all representative sets is 
$2^{O(\ddim)} n \log \Delta$.
It follows that the program has 
$2^{O(\ddim)} n \log \Delta$ 
constraints, each with only a constant number of non-zero variables.
As before, the program can be adapted to the framework of 
\citet{KY-14} by separating the $\max$ term into two separate 
constraints, and introducing dummy variables 
$\tilde{z}_i,\tilde{z}'_i,\tilde{z}''_i$
respectively satisfying 
$z_i + \tilde{z}_i = 1$,
$z'_i + \tilde{z}'_i = 1$ and 
$z''_i + \tilde{z}''_i = 1$.
The claimed runtime 
follows.
\end{proof}

\paragraph{Extension to the weak mean.}
In light of Corollary~\ref{cor:s<=wlogn},
relaxing the constraint 
in (\ref{eq:dist-reg})
from
$f \in \barLs_{\Lpr}(\xn,\rho,\mu_n)$
to
$f \in \barLw_{\Lpr}(\xn,\rho,\mu_n)$
will yield an improvement in the objective function
that can also be achieved
via the relaxation
$f \in \barLs_{2\Lpr{\log n}}(\xn,\rho,\mu_n)$,
and hence we forgo designing
a specialized algorithm for this case.

\subsection{Classification}
\label{sec:learn-class}

We show below that
the sample smoothing problem for classification under average slope constraints in the strong-mean sense admits an algorithmic solution, but this solution reduces to solving an NP-hard problem.
(This does not necessarily imply however that the smoothing problem in the strong-mean sense is NP-hard.)
Fortunately, we are able to produce
an efficient bi-criteria approximation algorithm
for the sample smoothing problem under average slope constraints in the 
{\em weak}-mean sense.
Given our current state of knowledge, 
the weak mean provides us an unexpected computational advantage
over the strong mean, 
in addition to its being a more refined indicator of average smoothness.

\paragraph{Smoothing under the strong mean.}

%%<AK_TEX_SCRIPT_CODE_DON'T_ALTER_OR_DUPLICATE|"class-strong.tex"
Let
$\joint$ be some distribution on $\X\times\set{0,1}$,
and
$S=(X_i,Y_i)_{i\in[n]}\sim\joint^n$ be a set sampled i.i.d. from $\joint$. 
At constant confidence level $\delta$,
the generalization bound
(\ref{eq:class-gen})
implies that
any $f:\X\to[0,1]$ with
$\bLas_f(\mu,\X)\le L$
that makes $k=
\sum_{i=1}^n \pred{f(X_i)\neq Y_i}
$ or fewer mistakes on the sample
will achieve,
with high probability,
a generalization error
\beqn
\label{eq:class-k-bd}
\P_{(X,Y)\sim\joint}( 
\Int{f(X)}
\neq Y )
\le
\frac{k}{n}
+G(L,n)
=:Q(f,L),
\eeqn
where $G(\cdot,\cdot)$
is the bound in the right-hand side of
(\ref{eq:class-gen}).

We wish to find a hypothesis approximately minimizing the bound $Q(\cdot,\cdot)$
in (\ref{eq:class-k-bd}).
An intuitive approach might involve solving the following problem, 
which we call the Minimum Removal Average Slope Problem:
Given an average slope target value $L$,
remove the smallest number points from $S$ so that the resulting point set
attains average slope at most $L$.
Clearly, an algorithm solving or approximating  
the Minimum Removal Average Slope Problem can be leveraged to find a 
minimizer for (\ref{eq:class-k-bd}). 
However, we can show that such an approach is algorithmically infeasible:

\begin{claim}
The Minimum Removal Average Slope Problem is NP-hard.
Assuming the Exponential Time Hypothesis (ETH),
it is hard to approximation within a factor $n^{1/ \log^r \log n}$
for some universal constant $r$.
\end{claim}
\begin{proof}
The hardness follows via a reduction from the Minimum $k$-Union Problem.
In this problem we are given a collection $C$ of $n$ sets and a parameter $k$, 
and must find a subset $C' \subset C$ of size $|C'|=k$
so that the union of all sets in $C'$ is minimized.
The Minimum $k$-Union Problem is known to be NP-hard \citep{CDKKR-16}, 
and under the ETH, it is hard to approximate the minimum union within a factor of
$n^{1/\log^{r'}\log n}$ for some universal constant $r'$. 
(The hardness of approximation follows directly from the Densest $k$-Subgraph Problem, which 
can be viewed as a special case of Minimum $k$-Union Problem \citep{M-17, C-20}.)

The reduction is as follows: Given an instance $(C,k)$ of Minimum $k$-Union, 
we create an instance of Minimum Removal Average Slope.
Create bipartite point set $S = S_e \cup S_s$ thus:
Set $S_e \in S$ has a point corresponding to each element 
in the element-universe of $C$.
Set $S_s \in S$ has a point of weight $m=|S_e|+1$
corresponding to each set in $C$.
(A point can be assigned weight $m$ by placing $m$ copies of the same point in $S_e$.)
For each point pair $s \in S_s, e \in S_e$
we set $\rho(s,e)$ equal to 2 if $e \in s$, and 1 otherwise.
Now let the target average slope be 
$\frac{2|S|-km}{|S|}$.
Clearly, this can only be attained by deleting the minimum number
of points in $S_e$ so that at least $k$ points of $S_s$
are not within distance 1 of any point of $S_e$.
This is equivalent to finding $k$ sets of $C$ of minimum union.
The reduction preserves hardness-of-approximation as well.
\end{proof}

One attempt around the hardness result would be to
mimic the approach taken for regression:
Identify a target error term $\eps \in [\frac{1}{n},1]$
(via binary search),
and remove from $S$ the ``worst'' $2\eps n$ points in order to minimize
the 
{\em maximum}
slope of the remaining points.
This in turn may be approximated using the algorithm of
\citet{DBLP:journals/tit/GottliebKK14+colt}, 
which runs in time 
$2^{\ddim} n \log n + \ddim^{O(\ddim)}n$.
Such an approach would yield a classifier achieving
a value of $Q(\cdot,\cdot)$
within a factor of $(1/\eps)^{O(\ddim)}$ of the optimal one.
A much better approximation 
factor
of $2^{O(d)}\log(n)$
is feasible, however, as we shall
see below.

\begin{theorem}[Training and generalization 
for
classification, weak mean]
\label{thm:weak-ave-lip-learn}
Let
$\joint$ be some distribution on $\X\times\{0,1\}$,
and
$S_n=(X_i,Y_i)_{i\in[n]}\sim\joint^n$ be a set sampled i.i.d. from $\joint$. 
Denote
by
$\hat f \in \barLw_{\Lpr}(\X_n,\rho,\mu_n)$ the minimizer of $
R(\cdot;\joint_n)
$.
Then there is an efficient learning algorithm $\mathcal{A}$, which constructs a classifier $f = \mathcal{A}(S_n)$ such that such that for any 
given $L>0$ and $0<\delta<1$:
\ben
\item[(a)]
With probability at least
$1
-2^{O(\ddim)} \log^3(n)/n
-3\delta$,
\beq
R(f;\joint) 
\leq 
2^{O(d)}\log(n)R(\hat f;\joint_n) 
+
\frac{C_\delta \sqrt{2^{O(d)}\log^3(n)L}}
{n^{1/8d}}
+
\frac{C_\delta^{-d/2} \sqrt{2}} {n^{5/16}}
+
3\sqrt{\frac{\log(2/\delta)}{2n}}
,
\eeq
\item[(b)]
$f(x)$ can be evaluated at each $x \in \X$ in time $O(n^2)$ 
after a one-time ``smoothing'' computation of 
$O(n^2) + 2^{O(d)} n \log^2 n$.
\een
where 
$C_\delta$ is a constant depending only on $\delta$. %
\end{theorem}
\begin{proof}
The bi-criteria approximation algorithm
in Lemma~\ref{lem:bicriteria} yields an
$\tilde f\in \barLw_{O(1)L}(\X_n,\rho,\mu_n)$
whose empirical risk is within a
$2^{O(d)}\log(n)$
factor of the optimal 
$R(\hat f;\joint_n)$
and the adversarial extension procedure
in Lemma~\ref{lem:adv-ext-class} for classification
guarantees
that the \extname\ extension of $\tilde f$ from $\X_n$ to $\X$
verifies 
$f\in \barLw_{
2^{O(d)}\log^3(n)L
}(\X,\rho,\mu)$
with high probability.
The generalization bound in (\ref{eq:class-gen}) then applies directly to yield (a).
The runtimes claimed in (b) are demonstrated in 
Remark~\ref{rem:PMSE-argmax} (which argues that \extname\ can be evaluated in time $O(n^2)$)
and the proof of Lemma~\ref{lem:bicriteria}.
\end{proof}  

\paragraph{Bi-criteria approximation for smoothing
under weak mean.} 
We wish to perform smoothing of $\bLaw_f(\mu_n,\X_n)$. For this,
we define the {\em continuous local slope removal problem} (CLSRP) as follows:
Let $(\X,\rho)$ be a metric 
space 
with $\diam(\X)\le 1$ and $\ddim(\X)<\infty$.
Given $S=(\xn,y_{[n]})
\in(\X^n,\{0,1\}^n)$
and $L>0$, relabel the minimal amount of points in
$S$ with any real label in $[0,1]$, 
so that for the resulting label-set the number of points with local 
slope $tL$ or greater is at most ${n}/{t}$ for all real 
$t \in [1,n+1]$.
Notice that solving CLSRP for a given $L$ implies that $\bLaw_f(\mu_n,\X_n) \leq L$. By definition of CLSRP, $k\mu_n(M_f(k)) \leq L$ for $L \le k \le (n+1)L$, while this extends trivially for all $k \le L$ and $k \ge (n+1)L$. 

Suppose that the solution $I$ of CLSRP consists of relabeling $k>0$ points in $S$. Then an $(a,b)$-bicriteria approximation ($a,b \ge 1$) 
to the solution of CLSRP on $I$ is one in which at most 
$ak$ points are relabeled, while the number of points with local
slope $btL$ or greater is at most ${n}/{t}$ for all real
$t \in [1,n+1]$.
We can show the following:

\begin{lemma}
\label{lem:bicriteria}
CLSRP admits a
$(2^{O(\ddim)} \log n, O(1))$-bicriteria approximation in time 
$O(n^2) + 2^{O(\ddim)} n \log^2 n$.
\end{lemma}

\begin{proof}
The construction is as follows. Let $t_i = 2^i$ for 
integer $i \in [0, \lceil \log n \rceil]$.
For each $t_i$ we will construct a points set $P_i$ of points to
be relabelled, and the final solution will be 
$P = \cup_i P_i$:

For each $i$, construct a $\frac{1}{2t_i L}$-net of $S$ called $T_i \subset S$. 
Associate every point in $S$ with its nearest neighbor in $T_i$,
and let the neighborhood of $p \in T_i$ ($N(p)$) include all points 
of $S$ associated with $p$.
Now, if not all points in $N(p)$ have the same sign, then create a 
new point $p' \in T_i$ which is a copy of $p$ but with label $1-l(p)$.
Remove from $N(p)$ all points with label $1-l(p)$,
and place them in $N(p')$ instead. (Note that $p$ is
found in $N(p)$, but $p'$ is not found in $N(p')$.)
This can all be done in time $2^{O(\ddim)} n \log n$ using a standard
point hierarchy.

Now create a new subset $T'_i \subset T_i$ thus:
$p \in T_i$ is added to $T'_i$ 
only if there is some point $q \in T_i$ 
with
$l(p) \ne l(q)$
satisfying 
$d(p,q) \le \frac{2}{t_i L}$.
We will show below that (roughly speaking) 
for $p \in T_i'$
all points in $N(p)$ have high local slope constant, 
while for $p \in T_i \setminus T_i'$
all points in $N(p)$ have low local slope constant.
We will associate a weight with each $p \in T_i'$ thus:
For all $p \in T_i'$, let $T_0(p), T_1(p)$
consist of all points of $T_i'$ 
within distance 
$\frac{2}{t_i L}$ of $p$,
and with respective labels 0,1.
Define
$S_0(p) = \cup_{q \in T_0(p)} N(q)$ 
and 
$S_1(p) = \cup_{q \in T_1(p)} N(q)$.
With each point $p \in T$ we associate weight
\[	w(p) = \min \{ |S_0(p)|, |S_1(p)| \}.	\]
Intuitively, this weight reflects the cost of reducing the
local slope constant of all points in $N(p)$;
this requires relabeling all points in either 
$S_0(p)$ or $S_1(p)$.

Let $m = \sum_{p \in T_i'} |N(p)|$.
Let $m' = m - \frac{6n}{t_i L}$
(a value which we will motivate below),
and we wish to find a minimal weight 
subset $C^*_i \subset T_i'$ satisfying that 
$\sum_{p \in C^*_i} |N(p)| \ge m'$.
This is a version of the NP-hard Minimum Knapsack Problem,
but we can find in time $O(n \log n)$ a
subset $C_i \subset T_i'$ satisfying that
$\sum_{p \in C_i} |N(p)| \ge m'$,
with $w(C_i) \le 2w(C^*_i)$ \citep{CFLZ-91}.
Then the set of points to be relabeled is
$P_i = \cup_{p \in C_i} N(p)$,
and as above the final solution is 
$P = \cup_i P_i$.
This completes the construction.

To prove correctness, first fix some $t_i$. 
Consider a pair $p,q \in S$ which are found 
in the neighborhoods of respective points $p',q' \in T$.
If $d(p,q) \le \frac{1}{t_i L}$, 
then by the triangle inequality 
\[
\rho(p',q') 
\le \rho(p,q) + \rho(p,p') + \rho(q,q')
\le \frac{1}{t_i L} + \frac{1}{2 t_i L} + \frac{1}{2 t_i L} 
=   \frac{2}{t_i L}.
\]
It follows that if
$\rho(p',q') > \frac{2}{t_i L}$
then
$\rho(p,q) > \frac{1}{t_i L}$.
Also, if
$\rho(p',q') \le \frac{2}{t_i L}$
then
\[
\rho(p,q) 
\le \rho(p',q') + \rho(p,p') + \rho(q,q')
\le \frac{2}{t_i L} + \frac{1}{2 t_i L} + \frac{1}{2 t_i L} 
=   \frac{3}{t_i L}.
\]

For the approximation bound on the number of relabelled points:
Consider some $p \in T_i'$.
By the above calculation,
if at least one point in each of 
$S_0(p)$ and $S_1(p)$
is not chosen for relabelling, 
then the above bound along with Corollary~\ref{cor:PMSE-prop} imply that
no point of $N(p)$ can attain local slope constant less than 
$\frac{1}{\frac{3}{t_i L} + \frac{3}{t_i L}}
= \frac{t_i L}{6}$.
Now, as $N(p) \subset S_0(p) \cup S_1(p)$,
and since the exact solution to CLSRP has slope constant 
$\frac{t_i L}{6}$
or greater on at most 
$\frac{6n}{t_i L}$
points, the exact solution must relabel all but at most
$\frac{6n}{t_i L}$ 
points of $\cup_{p \in T_i'} N(p)$.
And further, for any point $p$ relabelled by the exact solution
to achieve local slope constant $\frac{t_i L}{6}$ or less,
the exact solution must also relabel all of either
$S_0(p)$ or $S_1(p)$.
By the packing property (\ref{eq:ddim-pack}), for any $t_i$,
any point $q \in S$ appears in $2^{O(\ddim)}$
sets of the form $S_0(p),S_1(p)$,
and so it follows that the number of points relabelled by the above construction
for $t_i$ is within a factor $2^{O(\ddim)}$ of the number of points relabelled by
the exact solution.
Summing over $O(\log n)$ values of $i$, we have that the number of points
relabelled by the approximation algorithm is within a factor 
$2^{O(\ddim)} \log n$
of the exact solution.

For the bound on the local slope constant:
Fix some $t_i$.
For any $p \in T_i$,
if all points in $S_0(p)$ or $S_1(p)$ are relabelled according to \extname,
then as shown above the distance from any point 
$p' \in N(p)$ to a point of label 0 or 1 (respectively)
is greater than $\frac{2}{t_i L}$. 
Then by Corollary \ref{cor:PMSE-prop}(ii), 
the local slope constant of $p$ will be at most $\frac{t_i L}{2}$.
By construction, at most 
$\frac{6n}{t_i L}$ points remain with this slope constant,
and the result follows.
\end{proof}

%%<AK_TEX_SCRIPT_CODE_DON'T_ALTER_OR_DUPLICATE|"ext-alg.tex"

\section{Adversarial extension: regression}
\label{sec:adv-ext}

As we discussed in Section~\ref{sec:main-res} (and, in greater detail,
at the beginning of Section~\ref{sec:learn}), ensuring that a function
with on-average smooth behavior
on the sample also possesses this property on the whole space is non-trivial.
To this end, we 
introduce
the 
following 
{\bf adversarial extension} game.
First, $\X_n=\Xn\sim\mu^n$ is drawn,
which induces the usual empirical measure $\mu_n$.
Next, the adversary picks
an $\e>0$ and $y:\X_n\to[0,1]$ arbitrarily.
Finally,
the
learner is challenged to construct a function
$f:\X\to[0,1]$ satisfying the following criteria:
\ben
\item[(a)] 
$f$ is close to $y$ on the sample
\ben
\item[(i)] (w.r.t.
  strong average slope):
$
\nrm{f-y}_{\Lp{1}(\mu_n)} \le O(\e) 
(1\vee
\bLas_y(\mu_n,\X_n)
)
,
$
\item[(ii)] (w.r.t.
  weak average slope):
$
\nrm{f-y}_{\Lp{1}(\mu_n)} \le \tilde O(\e) (1\vee\bLaw_y(\mu_n,\X_n));
$
\een
\item[(b)] $f$'s average slope does not much
exceed the sample one
\ben
\item[(i)] (w.r.t.
  strong average):
$
\bLas_f(\mu,\X) \le O(1)
\bLas_y(\mu_n,\X_n),
$
\item[(ii)] (w.r.t.
weak average):
made precise in Lemma~\ref{lem:b.ii}.
\een
\een

Two immediate observations are in order. First, we notice the tension
between the 
criteria (a) and (b). Each one is trivial to satisfy individually
--- (b) by a constant function and (a) by {\em any} proper extension, including \extname\ --- but it is not obvious that both can be satisfied simultaneously. Second, (b) can at best hold with high probability.
Indeed, let $\mu$ be a distribution over $\X=[0,1]$ with high
density near $1/2$ and low density at the endpoints. Suppose further that the sample $\X_n$ has turned out rather unrepresentative:
many points near the endpoints and only two near $1/2$.
In such a setting, the adversary can force, say,
$\bLas_f(\mu,\X)/\bLas_y(\mu_n,\X_n)$
to be large for {\em any} extension $f$ of $y$.

Throughout this section, we assume
$\diam(\X)\le1$
and $d:=\ddim(\X)<\infty$.

\subsection{Proving (a.i) and (b.i), strong mean}
We begin by handling the technically
simpler case of addressing the adversarial
extension problem for the strong mean.

\begin{lemma}[Adversarial extension for strong mean]
\label{lem:adv-ext-strong}
In the adversarial extension game,
there is an efficient algorithm for
satisfying conditions (a.i) and (b.i),
the latter with
probability at least
\beq
1-\exp\left( {-n (\e/8) ^ {d+1} + d\log((8/\e)} \right),
\eeq
where 
$\e<1$
and
$n\ge(8/\e)^{d+2}$.
(For concrete constants, $O(\e)$ in (a.i) may
be replaced by $3\e$ and $O(1)$ in (b.i) by $5$.)

\end{lemma}

\newcommand{\enet}{V}
\newcommand{\Xe}{\X_n(\e)}
\newcommand{\Xec}{\X_n'(\e)}

\begin{proof}
We begin by constructing the extension $f$.
\begin{enumerate}
\item
Sort the $\La_y(X_i,\X_n)$ in decreasing order,
and let $\Xe\subset\X_n$ consist of the $\floor{\e n}$ points
with the largest values (breaking ties arbitrarily).
\item
Put $\Xec:=\X_n\setminus\Xe$.
\item
Let $V$ be an $\e$-net of $\Xec$.
\item
Define $f:\X\to\R$ as the \extname\ extension
of $y$ from $\enet$ to $\X$,
as
defined in Definition~\ref{def:PMSE}.
\end{enumerate}
Since $|\Xec|< n$, the value of $f(x)$ can be computed in time $O(n^2)$
at any given $x\in\X$.
The 
computational runtime for net construction is
$\min \{ 2^{O(d)}n\log(n) \log \Delta,O(n^2)\}$
\citep{KL04}.

\paragraph{Proof of (a.i).
}
Recalling that $f\equiv y$ on $\enet$,
we 
have
\beqn
\nonumber
\nrm{f-y}_{\Lp{1}(\mu_n)} &=&
\frac{1}{n}\sum_{x\in\X_n\setminus\enet}|f(x)-y(x)|
\\
&=&
\label{eq:a.i-decomp}
\frac{1}{n}\sum_{x\in
\Xe
\setminus\enet
}|f(x)-y(x)|
+
\frac{1}{n}\sum_{x\in\Xec
\setminus\enet
}|f(x)-y(x)|
.
\eeqn
Since $0\le f,y\le 1$,
the first term is trivially bounded by 
$|\Xe|/n\le\e$.
To bound the second term, we
recall the map $\phi_V:\Xe\to V$
defined in Section~\ref{sec:def}
---
and in particular, that
$\rho(x,\phi_V(x))\le\e$
---
and
compute
\beqn
\frac{1}{n}\sum_{x\in\Xec
\setminus\enet
}|f(x)-y(x)|
&\le&
\frac{1}{n}\sum_{x\in\Xec\setminus\enet}\frac{
  \e
}{\rho(x,\phi_V(x))}|f(x)-y(x)| \nonumber\\
&\le&
\frac{\e}{n}
\sum_{x\in\Xec\setminus\enet}
\frac{|y(\phi_V(x))-y(x)| + 
|f(x) - y(\phi_V(x))|}{\rho(x,\phi_V(x))} \nonumber\\
&\le&
\frac{\e}{n}
\cdot 2
\sum_{x\in\Xec\setminus\enet}
\La_y(x,\X_n)
\label{eq:a.i-2}
\\
\nonumber
&\le&
2\e
\bLas_y(\mu_n,\X_n)
,
\eeqn
where the third inequality follows from Theorem~\ref{thm:PMSE}:
\beq
\frac{|f(x) - y(\phi_V(x))|}{\rho(x,\phi_V(x))}
=
\frac{|f(x) - f(\phi_V(x))|}{\rho(x,\phi_V(x))}
\le
\La_f(x,V)
\le
\La_y(x,V)
\le
\La_y(x,\X_n)
.
\eeq

Hence $f$ satisfies (a.i) with $3\e$.

\paragraph{Proof of (b.i).}
 Let 
 $q>0$ be a parameter
 to be determined later.
 Let $U\subseteq\X$ be an $\eps/4$-net of $\X$,
 with induced Voronoi partition
$\Pi=\set{ \phi_U\inv(u) : u\in U}$.
Put $m=|\Pi|\le (8/\eps)^{d}$ 
(see (\ref{eq:ddim-pack}))
and segregate the elements of $\Pi$ into 
``light,'' $\Pi_0$, and ``heavy,'' $\Pi_1$:
\beq
\Pi_0:=\set{B\in\Pi: 
\mu_n(B)
< nq/m},
\qquad
\Pi_1:=\set{B\in\Pi: 
\mu_n(B)
\ge nq/m}
.
\eeq

We will need three auxiliary lemmata (whose proof is deferred to the Appendix).

\begin{lemma}[Local slope smoothness of the \extname]
\label{lem:pose-smooth}
Suppose that $A\subseteq\X$ and $f$ is the \extname\ of
some function from $A$ to $\X$. Suppose further that
$E\subseteq\X$ satisfies
\beqn
\label{eq:E}
\diam(E) &\leq& 
\frac{1}{2}\min_{x\neq x'\in A} \rho(x,x').
\eeqn
Then
\beq
\sup_{x,x'\in E}
\frac{\La_{f}(x,\X)}{\La_{f}(x',\X)} \leq 2.
\eeq
\end {lemma}

\begin{lemma}[Accuracy of empirical measure]
\label{lem:cell-rec}
The individual heavy cells have empirical measure within a constant
factor of their true measure, with high probability:
\beqn
\label{eq:heavy-big}
\P\paren{
\min_{B\in\Pi_1}\frac{\mu_n(B)}{
\mu(B)
}\le \frac{1}{2}
} &\le& m\exp(-nq/4m),\\
\label{eq:heavy-small}
\P\paren{
\max_{B\in\Pi_1}\frac{\mu_n(B)}{
\mu(B)
}\ge 2
} &\le& m\exp(-nq/3m).
\eeqn
Additionally, the combined $\mu$-mass of the light cells is
not too large:
\beqn
\label{eq:light-big}
\P\paren{
\sum_{B\in\Pi_0}\mu(B)\ge 2q
}
&\le&
\exp\sqprn{-(m+nq^2)/2+q\sqrt{mn}},
\qquad
nq^2\ge m.
\eeqn

\end{lemma}

Finally, we bound the 
Lipschitz constant of the \extname\ $f$
of $y$:
\begin{lemma}[Lipschitz constant of $f$]
\label{lem:max-alg-lip}
\beq
\lip{f} &\le& 
2\e\inv \bLaw_y(\mu_n,\X_n)
\le
2\e\inv \bLas_y(\mu_n,\X_n)
.
\eeq
\end{lemma}

Armed with these results,
we are
in a position to prove (b.i).
We choose $q:=\eps/8$
and calculate
\beqn
 \bLas_f(\mu,\X) &=& \int_{\X} \La_{f}(x,\X)\d\mu 
\;=\; 
\sum_{B \in \Pi}\int_{B} \La_{f}(x,\X)\d\mu \nonumber\\
 &=& \sum_{B \in \Pi_0}\int_{B} \La_{f}(x,\X)\d\mu 
 +
\sum_{B \in \Pi_1}\int_{B} \La_{f}(x,\X)\d\mu.
\label{eq:Pi01}
\eeqn
The first term in (\ref{eq:Pi01}) is bounded using
(\ref{eq:light-big}) in Lemma~\ref{lem:cell-rec}
and Lemma~\ref{lem:max-alg-lip}
\beq
 \sum_{B\in\Pi_0}\int_{B}\La_{f}(x,\X)\d\mu &\leq& 
 \sum_{B\in\Pi_0}\int_{B}\frac{2
 \bLas_y(\mu_n,\X_n)
 }{\eps} \d\mu \\
 &=& 
 \frac{2
 \bLas_y(\mu_n,\X_n)
 }{\eps}
 \sum_{B\in\Pi_0}\mu(B)\\
&\le_p& 
\frac{2
\bLas_y(\mu_n,\X_n)
}{\eps}\left(2\cdot\frac{\eps}{8}\right) \le
\bLas_y(\mu_n,\X_n)
,
\eeq
where the inequality indicated by $\le_p$
holds with high probability, as in (\ref{eq:light-big}).

To
bound the second term in (\ref{eq:Pi01}),
we first observe that for all $B \in \Pi$ and $x' \in B$,
\beq
\La_{f}(x',\X)
&\le&
2\min_{x\in\X_n\cap B}
\La_{f}(x,\X)
.
\eeq
Indeed, this follows from Lemma~\ref{lem:pose-smooth}, invoked with
$A=V$ and $E=\X_n\cap B$.
Proceeding,
\beq
\sum_{B\in\Pi_1}
\int_{B} \La_{f}(x,\X)d\mu 
&\le& 
\sum_{B\in\Pi_1}\int_{B} 
2\min_{x\in\X_n\cap B}
\La_{f}(x,\X)
\d\mu \\
&=& 
\sum_{B\in\Pi_1}
2\min_{x\in\X_n\cap B}
\La_{f}(x,\X)
\mu(B)\\
&\le_p& 
4
\sum_{B\in\Pi_1}
\min_{x\in\X_n\cap B}
\La_{f}(x,\X)
\mu_n(B)
\\
&=& 
\frac{4}{n}
\sum_{B\in\Pi_1}
\sum_{x'\in \X_n\cap B}
\min_{x\in\X_n\cap B}
\La_{f}(x,\X)
\\
&\le& 
\frac{4}{n}
\sum_{B\in\Pi_1}
\sum_{x'\in \X_n\cap B}
\La_{f}(x',\X)\\
&\le& 
\frac{4}{n}
\sum_{x'\in \X_n}
\La_{f}(x',\X)
\le 4 \bLas_y(\mu_n,\X_n)
,
\eeq
where the inequality indicated by $\le_p$
holds with high probability, as in
Lemma~\ref{lem:cell-rec},
and the last inequality follows from
Corollary~\ref{cor:PMSE-prop}(i),
since $\La_f(x',\X)$ is determined by
$f$'s values on
$\enet$:
\beq
\La_{f}(x',\X) = 
\La_{f}(x',\enet) = 
\La_{f}(x',\X_n) \le 
\La_{y}(x',\X_n).
\eeq

Plugging these estimates back into (\ref{eq:Pi01}) yields
(b.i) with $5$ in place of $O(1)$.
Applying the union bound to the probabilistic inequalities marked by $\le_p$ above yields the claim with probability at least $1-\delta$,
where
\beq
\delta=
\left[
\exp
\left(
-\frac{m+nq^2}{2}-q\sqrt{mn}
\right)
+
m
\exp
\left(
-\frac{nq}{4m}
\right)\right]
\le
\exp\left( {-n (\e/8) ^ {d+1} + d\log((8/\e)} \right).
\eeq
\end{proof}

\subsection{Proving (a.ii) and (b.ii), weak mean}
As discussed at the end of Section~\ref{sec:learn-reg},
our training procedure for regression obtains
comparable results for both strong- and weak-mean regularization.
Hence, only the former is fleshed out, and the latter,
corresponding to
claims (a.ii) and (b.ii) above,
is not directly invoked in
this paper. We find the proof of (a.ii) and (b.ii) to be of independent
interest,
and present it in
Section~\ref{sec:adv-ext-regr-weak-mean}.

%%<AK_TEX_SCRIPT_CODE_DON'T_ALTER_OR_DUPLICATE|"ext-alg-class.tex"
\section{Adversarial extension: classification}
\label{sec:adv-ext-class}

The adversarial extension for classification differs
in several aspects
from its regression
analogue 
in Section~\ref{sec:adv-ext}.
Conceptually, there is a ``type mismatch'' between a $[0,1]$-valued function 
and the $\set{0,1}$ labels it is supposed to predict. The actual prediction is performed
by rounding via $f(x)\mapsto \Int{f(x)}$, but the sample risk
charges a unit loss for every $f(x_i)\neq y_i$, regardless
of how close the two might be.\footnote{
A simple no-free-lunch argument shows that one could not hope to obtain a
generalization bound with sample risk based on the $\Int{f(x_i)}\neq y_i$ loss.
Indeed, the function $f(x_i)=1/2+\eps y(x_i)$ would achieve zero empirical risk while
also having an arbitrarily small Lipschitz constant.
}
Thus, for adversarial extension,
no distortion of the adversary's labels $y$ is allowed --- the only
changes the learner
makes to the sample labels occur during the smoothing procedure in Section~\ref{sec:learn-class}
---
and hence there is no $\eps$ parameter.
The strict adherence to $y$
incurs the cost of a $2^{O(\ddim)}\polylog n$ increase in the average slope of the extension (unlike
in regression, where a small distortion of $y$ afforded an at most constant increase).
Finally, note that though intermediate results for the strong mean are obtained, only those for the weak
mean are algorithmically useful, in light of the hardness result in Section~\ref{sec:learn-class}.

\begin{lemma}[Adversarial extension for classification]
\label{lem:adv-ext-class}
Suppose that $\X_n=\Xn\sim\mu^n$
and
$y:\X_n\to\{0,1\}$ verifies $\bLaw_y(\mu_n,\X_n) \leq n$,
but is otherwise arbitrary.
Then
there is an efficient algorithm for
computing a function $f:\X\to[0,1]$ that coincides with $y$ on $\X_n$
and satisfies
\beqn
\nonumber
\bLaw_f(\mu,\X)
&\le&
\bLas_f(\mu,\X)\\
\label{eq:adv-class}
&\le& 2^{O(\ddim)}\log^2(n) ~\bLas_y(\mu_n,\X_n)\\
\nonumber
&\le& 2^{O(\ddim)}\log^3(n) ~\bLaw_y(\mu_n,\X_n).
\eeqn
with
probability at least
\beq
1-2^{O(\ddim)} \log^3(n)/n.
\eeq
\end{lemma}
\paragraph{Remark.}
The assumption $\bLaw_y(\mu_n,\X_n) \le n$ incurs no loss of generality
because $\bLaw_f(\mu,\X)>n$ yields vacuous generalization bounds.

\newcommand{\myxi}{x}
\newcommand{\myxj}{x'}
\newcommand{\ball}{B}
\newcommand{\balls}{\mathcal{B}}
\newcommand{\UB}{\tilde {\mathcal C}}
\newcommand{\locd}{d}

\begin{proof}
  Only the estimate in (\ref{eq:adv-class}) requires proof;
  the rest hold everywhere (and in particular,
  with probability $1$) by
  (\ref{eq:blasw}) and Corollary~\ref{cor:s<=wlogn},
  respectively.
  Our algorithm computes $f$ as the \extname\ extension of $y$ from $\X_n$ to all of $\X$. It remains to show that (\ref{eq:adv-class}) holds with the claimed probability. Throughout the proof, $d:=\ddim(\X)$.

Let $\{H_{2^{-i}}\}_{i=0}^{\lceil2\log{2n}\rceil}$ be a point hierarchy for $\X$ and to each net-point $p \in H_{r}$
associate a ball $B(p, r)$. Note that the balls associated with points in the lowest level of the hierarchy have a radius of $r \leq \frac{1}{4n^2}$.
Hence, any ball $\ball$ associated to some point in the lowest level in the hierarchy contains $y$-{\em homogeneously} labeled points of $\X_n$.
That is, for any $\myxi,\myxj \in\ball$ we have $y(\myxi) = y(\myxj)$, otherwise contradicting the assumption that 
$\bLaw_y(\mu_n,\X_n) \leq n$. Furthermore, the nearest opposite-label point is at least at a distance of $1/n^2$ from any 
point in $\ball$; we will refer to this property as the {\em extended monochromatic property}.
Denote by $\balls_j$ the set of balls corresponding to points of $H_j$
(note that $j=2^{-i}$ is typically not an integer).

For any $x\in\X$, denote by $p_x \in \underset{p \in \X_n}\argmin \: \rho(p,x)$ a nearest neighbor of $x$
within the set $\X_n$. By property (ii) of corollary~\ref{cor:PMSE-prop}, we know that
\beqn
\label{eq:Lafpx}
\La_f(x,\X) = \frac{|f(x)-f(p_x)|}{\rho(x,p_x)} \leq \La_f(p_x,\X_n)
.
\eeqn

Let $i(x)$ be the minimum between the following: the smallest power of $1/2$ such that $(1/2)^{i(x)} \leq \rho(x,p_x)/4$ and $\lceil2\log{2n}\rceil$.
Notice that in either case, the inequality 
\beqn
\label{eq:ball_rad_min}
\rho(x,p_x)/8 \leq (1/2)^{i(x)}
\eeqn
holds. Let 
$\ball(x) \in \balls_j$ 
be the ball of radius $(1/2)^{i(x)}$ covering $x$. By the triangle inequality and the extended monochromatic property of the lowest level balls, for any $x' \in \ball(x)$ we have that
\beqn
\label{eq:ballsmoothness}
\frac{1}{2}\La_f(x,\X) \leq \La_f(x',\X) \leq 2\La_f(x,\X).
\eeqn

Now let $\X(i) \subseteq \X$ consist of all points $x \in \X$
for which $\ball(x)$ has radius $(1/2)^i$, 
and let 
$$
\tilde{\balls}_{(1/2)^i} =
\set{\ball(x)\in\balls_{(1/2)^i}:
  x\in\X(i)}
$$
and 
$\tilde{\balls} = \cup_j \tilde{\balls}_j$.
We claim that 
$|\tilde{\balls}| \le 2^{O(\locd)} n \log n$.
Indeed,
for any $p \in \X_n$, let 
$N(p,i) \subseteq \X(i)$
include every point $x \in \X(i)$ for which
$p_x = p$, and define
$$\tilde{\balls}_{(1/2)^i}(p) =
\set{
  \ball(x)\in \tilde{\balls}_{(1/2)^i}:
  x \in N(p,i)
}
$$
and
$\tilde{\balls}(p) = \cup_j \tilde{\balls}_j(p)$.
Since balls in $\tilde{\balls}_{(1/2)^i}(p)$ have radius
$(1/2)^i$ and
their centers
are within distance $8(1/2)^i$ of $p$, the
packing property
(\ref{eq:ddim-pack})
gives that 
$|\tilde{\balls}_{(1/2)^i}(p)| = 2^{O(\locd)}$
and so
$|\tilde{\balls}(p)| = 2^{O(\locd)} \log n$.
It follows that
$
|\tilde{\balls}| 
\le \sum_{p \in \X_n} |\tilde{\balls}_j(p)|
= 2^{O(\locd)} n \log n
.
$

\newcommand{\ntag}{n}
\newcommand{\cE}{\mathcal{E}}
\newcommand{\Xnj}{\X_{n\setminus\myxj}}
\renewcommand{\myxj}{p}

We would like to claim that the empirical measure of each
$\ball\in\tilde\balls$ is close to its true measure.
Some care must be taken here, since $\tilde\balls$ is itself
a random set, determined by the same sample that determines
the empirical measure. However, conditional on any given
$p\in\X_n$ --- which determines the set $\tilde\balls(p)$ ---
we can use the remaining
$
n-1$ sample points to
estimate the mass of each $\ball\in\tilde\balls(p)$.
To avoid the notational nuisance of distinguishing
fractions involving $n$ and $n-1$, we will write
$a\approx b$ to mean $a=(1+\Theta(1/n))b$
and
$a\lesssim b$ to mean $a=(1+O(1/n))b$
for the remainder of the proof.

For any $\ball\in\tilde{\balls}$, let
$\cE_\ball$ be the event that
\beq
\mu(\ball) \gtrsim \frac{2\log{\ntag}}{\ntag}
&\implies&
  \mu(\ball)\lesssim   \mu_n(\ball)\log n
  \eeq
  (in words: if $\ball$ is sufficiently ``heavy'' then it is not under-sampled).
It follows from Theorem~\ref{thm:mult-chernoff}
that $\P(\cE_\ball)\ge1-
O
\paren{\frac{\log^2 \ntag}{\ntag^2}}$
holds for each
$\ball\in\tilde{\balls}$.
We argued above that
$|\tilde{\balls}(p)| = 2^{O(\locd)} \log n$,
whence
the event
$\cE(p):=\cap_{\ball\in\tilde\balls(p)}\cE_\ball$,
conditional on the given $p\in\X_n$,
occurs with probability at least
$
1 -
\frac{2^{O(\locd)}\log^3n}{n^2}
.  
$
Finally, taking a union bound over the
$n$ draws of $\X_n$,
we have that
\beq
\cE :=\bigcap_{p\in\X_n}\cE(p)
\eeq
holds with probability at least
$
1 -
\frac{2^{O(\locd)}\log^3n}{n}
.  
$
The remainder of the proof proceeds conditionally on the high-probability
event $\cE$.

Let $\UB_0$ be the set of $B \in \tilde{\balls}$ such that 
$\mu_n(B)=0$
and $\UB_1 = \tilde{\balls} \setminus 
\UB_0$. Then
\beqn
\label{eq:U12}
\int_{\X}\La_f(x,\X)\d\mu
&\leq&
\sum_{\ball\in \UB_0}\int_{\ball}\La_f(x,\X)\d\mu
+ 
\sum_{\ball\in \UB_1}\int_{\ball}\La_f(x,\X)\d\mu
.
\eeqn

We begin by bounding the first term in (\ref{eq:U12}).
Since $\mu_n(\ball)=0$ for each $\ball\in\UB_0$ and we are assuming
event $\cE$, it follows that each of these balls must verify
$\mu(\ball) \lesssim 2\log(n)/n$.
Recalling the definition of $\tilde\balls(p)$ for $p\in\X_n=\set{x_1,\ldots,x_n}$,
\beq
\sum_{\ball\in \UB_0}
\int_{B}\La_f(x,\X)\d\mu
&=&
\sum_{i\in[n]}
\sum_{\ball\in\tilde\balls(x_i)\cap\UB_0}
\int_{B}\La_f(x,\X)\d\mu
\\
&\le&
\sum_{i\in[n]}
\sum_{\ball\in\tilde\balls(x_i)\cap\UB_0}
\int_{B}\La_f(x_i,\X_n)
\d\mu
\qquad
\text{(by (\ref{eq:Lafpx}))}
\\
&\le&
\sum_{i\in[n]}
\sum_{\ball\in\tilde\balls(x_i)\cap\UB_0}
\int_{B}\La_y(x_i,\X_n)
\d\mu
\qquad
\text{(pointwise optimality of \extname)}
\\
&\lesssim&
\frac{2\log n}{n}
\sum_{i\in[n]}
\sum_{\ball\in\tilde\balls(x_i)}
\La_y(x_i,\X_n)
\qquad
\text{(because
$\mu(\ball) \lesssim 2\log(n)/n$
  )}
\\
&\le&
2^{O(\locd)}\log^2{n}
\cdot\frac1n
\sum_{i\in[n]}
\La_y(x_i,\X_n)
\qquad
\text{(because
$|\tilde{\balls}(x_i)| = 2^{O(\locd)} \log n$
  )}
\\
&=&
2^{O(\locd)}\log^2(n)
~
\bLas_y(\mu_n,\X_n)
.
\eeq

We now proceed to bound the second term in (\ref{eq:U12}).
To this end, we analyze two possibilities: either a $\ball\in\UB_1$
is ``light,'' meaning that 
$\mu(\ball)\lesssim 2\log(\ntag)/\ntag$,
or else ``heavy,'' meaning that
$\mu(\ball)\gtrsim 2\log(\ntag)/\ntag$.
Now $\mu_n(\ball)>0\implies\mu_n(\ball)\ge1/n$,
and so for any light ball we have, by construction,
$\mu(\ball)\lesssim2\log(n)\mu_n(\ball)$.
On the other hand, conditional on event $\cE$, a heavy ball
satisfies
$\mu(\ball)\lesssim \log(n)\mu_n(\ball)$.

\beq
\sum_{\ball\in \UB_1}\int_{\ball}\La_f(x,\X)\d\mu
&\leq&
\sum_{\ball \in \UB_1}
\int_\ball
2\min_{x' \in\ball\cap\X_n}\La_f(x',\X)\d\mu
\qquad\text{(by (\ref{eq:Lafpx}) and (\ref{eq:ballsmoothness}))}
\\
&\lesssim&
2\log{n}\sum_{\ball \in \UB_1}
\mu_n(\ball)
\min_{x' \in\ball\cap\X_n}\La_f(x',\X_n)
\\
&=&
2\log{n}\sum_{\ball \in \UB_1}\frac{1}{n}
\sum_{\myxi \in\ball\cap\X_n}\min_{x' \in\ball\cap\X_n}\La_f(x',\X_n)
\\
&\leq&
2\log{n}\sum_{\ball \in \tilde\balls}\frac{1}{n}
\sum_{\myxi \in\ball\cap\X_n}\La_f(\myxi,\X_n)
\\
&\leq&
2^{O(\locd)}\log^2(n)
\frac{2}{n}
\sum_{\myxi \in \X_n}\La_f(\myxi,\X_n)
\qquad\text{(because $|\tilde\balls|\le2^{O(\locd)}\log n$)}
\\
&\leq&
2^{O(\locd)}\log^2(n)
\frac{2}{n}
\sum_{\myxi \in \X_n}\La_y(\myxi,\X_n)
\\
&=&
2^{O(\locd)}\log^2(n)
\bLas_y(\mu_n,\X_n).
\eeq

\end{proof}

%%<AK_TEX_SCRIPT_CODE_DON'T_ALTER_OR_DUPLICATE|"ackn.tex"
\paragraph{Acknowledgements.}
We thank Luigi Ambrosio and Ariel Elperin
for very helpful feedback on earlier attempts to define a notion of average smoothness,
and to Sasha Rakhlin for useful discussions.
Pavel Shvartsman and Adam Oberman were very helpful in placing
\extname\ in proper historical context.

\bibliography{refs}
\bibliographystyle{abbrvnat}

\appendix

%%<AK_TEX_SCRIPT_CODE_DON'T_ALTER_OR_DUPLICATE|"misc-ineq-notat.tex"
\section{Miscellaneous inequalities and notations}
\label{sec:misc}

\paragraph{Numerical inequalities, $\vee$, $\wedge$.}
We will use the following elementary facts:
for all $a,b,c\in\R$, we have
\beqn
\label{eq:abc}
(a+b+c)^2&\le&
3a^2+3b^2+3c^2
\eeqn
and
\beqn
\label{eq:abc1/2}
|a-b|\vee|a-c|
&\ge& \frac12|b-c|;
\eeqn
for all $a,b,c,d\in\R_{+}$ such that $\frac{a}{c} \neq \frac{b}{d}$, we have
\beqn
\label{eq:abcd}
\frac{a+b}{c+d} 
&>& \frac{a}{c} \wedge \frac{b}{d},
\eeqn
and
for or all $f:\X\to\R$ and $x,y,z\in\X$
such that $
f(x)
\le
f(y)
\le
f(z)
$, we have
\beqn
\label{eq:fxyz}
\frac{f(z)-f(x)}{\rho(z,x)}
\ge
  \frac{f(y)-f(x)}{\rho(y,x)}
  \wedge
  \frac{f(z)-f(y)}{\rho(z,y)}
  ,
\eeqn
where $s\vee t:=\max\set{s,t}$ and $s\wedge t:=\min\set{s,t}$.
The floor $\floor{\cdot}$ and ceiling $\ceil{\cdot}$ functions
map a real number $t$ to its closest integers below and above, respectively.

\paragraph{Bound on $\nrm{\mu-\mu_n}_1$.}
If $\mu$ is a probability measure
with support size $m$ and $\mu_n$
is its empirical realization,
then the following bound is well-known
(see, e.g., \citet[Eqs. (5) and (17)]{Berend2013}):
\beqn
\label{eq:TV-bk}
\norm{\mu-\mu_n}_1 
\le 
\sqrt{\frac{m}{n}}+\sqrt{\frac2n\log\frac2\delta},
\qquad 0<\delta<1
\eeqn
holds with probability at least $1-\delta$.

\paragraph{Order of magnitude.}
We use standard order-of-magnitude
notation $f=O(g)$ to mean that
$0\le f(\cdot)\le c g(\cdot)$ for some universal $c>0$. 
We write $f=\Theta(g)$ to indicate that both $f=O(g)$
and $g=O(f)$ hold.
In the tilde 
notation
($\tilde O(\cdot)$,
$\tilde \Theta(\cdot)$),
logarithmic factors are ignored.

%%<AK_TEX_SCRIPT_CODE_DON'T_ALTER_OR_DUPLICATE|"PMSE.tex"

\section{Pointwise Minimum Slope Extension}
\label{sec:PMSE}
As mentioned in Related work, the material in this section
turns out to have been largely anticipated
by \citet{MR2431047}
and is included here for self-containment
and uniformity of notation and terminology. The term \extname\ is ours,
and the pointwise minimality property of this extension was not explicitly
mentioned in \citet{MR2431047} (though is easily derivable from the results
presented therein).

Let $f:\X\to\R$, $x\in\X$ and $\emptyset\neq A\subseteq\X$ be fixed (and hence frequently suppressed in the notation for readability).
We assume for now that $|A|\ge 2$; the degenerate case $|A|=1$ will be handled below.
For $u,v\in A$,
define
\beq
R_x(u,v)&:=&\frac{f(v)-f(u)}{\rho(x,v)+\rho(x,u)},\\
F_x(u,v)&:=& f(u)+R_x(u,v)\rho(x,u)
,\\
R^*_x
&:=&\sup_{u,v\in A}R_x(u,v),\\
W_x(\eps) &:=&\set{ (u,v)\in A^2 :
R_x(u,v)>R^*_x-\eps
},
\qquad 0<\eps<R^*_x\\
\Phi_x(\eps) &:=& \set{F_x(u,v): (u,v)\in W_x(\eps)}.
\eeq
The assumption $|A|\ge2$ implies that $R^*_x>0$.
Further, $\sup_{x\in\X}R^*_x<\infty$ if and only if $\lip{\evalat{f}{A}}<\infty$.

\begin{definition}[\extname]
\label{def:PMSE}
For any metric space
$(\X,\rho)$, any $f:\X\to\R$
and $\emptyset\neq A\subseteq\X$
with
$\lip{\evalat{f}{A}}<\infty$
and
$
\diam(A)
\wedge
\ninf{\evalat{f}{A}}
<\infty$, define the {\bf Pointwise Minimum Slope Extension (\extname)} of $f$ from $A$ to $\X$,
denoted $f_A:\X\to\R$, by
\beqn
\label{eq:PMSE-def}
f_A(x) := 
\lim_{\eps\to0}\sup(\Phi_x(\eps)) = \lim_{\eps\to0}\inf(\Phi_x(\eps)),
\qquad x\in\X.
\eeqn
In the degenerate case $A=\set{a}$, define $f_A(x):=f(a)$.
\end{definition}
The first order of business is to verify that
\extname\ is well-defined (i.e., that the limit in
(\ref{eq:PMSE-def}) indeed exists):

\begin{lemma}
\label{lem:PMSE-argmax}
Assume $
A\subseteq\X$, $|A|\ge2$, and
$\lip{\evalat{f}{A}}<\infty$.
Then, for $(u,v)$ and $(u',v')$ in $W_x(\eps)$,
$\eps<R^*_x/2$,
we have
\beq
\abs{F_x(u,v)-F_x(u',v')}
&\le&
\eps
\min\set{
4\diam(A)
,
16{\ninf{\evalat{f}{A}}}/{R^*_x}
}
.
\eeq
\end{lemma}
\noindent{\bf Remark.}
In words, it suffices for either $A$ to be bounded
or for $f$ to be bounded on $A$ in order that
approximate maximizers of $R_x(\cdot,\cdot)$ all yield approximately the same
value of $F_x(\cdot,\cdot)$.
\begin{proof}
Since
$x$ is fixed, we omit it from the subscripts for readability.
Observe that 
$R(u,v)>0$
for $(u,v)\in W(\eps)$
and
that 
$F(u,v)=
f(v) - R(u,v)\rho(x,v)$.
Thus,
\beqn
\label{eq:fuFuvfv}
f(u) \le F(u,v) \le f(v),
\eeqn
and the same holds for $(u',v')\in W(\eps)$.
There is no loss of generality in assuming
$F(u,v) \le F(u',v')$.
In this case, (\ref{eq:fuFuvfv}) implies that
$f(u) \le f(v')$ and hence
\beq
R^* \geq \frac{{f(v')-f(u)}}{\rho(x,v')+\rho(x,u)}
&=&
\frac{f(v')-F(u',v')+F(u',v')-F(u,v)+F(u,v)-f(u)}{\rho(x,v')+\rho(x,u)}\\
&=&
\frac{f(v')-F(u',v')+F(u,v)-f(u)}{\rho(x,v')+\rho(x,u)} + \frac{F(u',v')-F(u,v)}{\rho(x,v')+\rho(x,u)}\\
&=&
\frac{R(u',v')\rho(x,v')+R(u,v)\rho(x,u)}{\rho(x,v')+\rho(x,u)} + \frac{F(u',v')-F(u,v)}{\rho(x,v')+\rho(x,u)}\\
&\geq&
\frac{R(u',v')\rho(x,v')+R(u,v)\rho(x,u)}{\rho(x,v')+\rho(x,u)} + \frac{{F(u',v')-F(u,v)}}{2\diam(A)}\\
&\ge&
\frac{R(u',v')\rho(x,v')+(R(u',v')-\e)\rho(x,u)}{\rho(x,v')+\rho(x,u)} + \frac{{F(u',v')-F(u,v)}}{2\diam(A)}\\
&=&
R(u',v') -\e +\frac{{F(u',v')-F(u,v)}}{2\diam(A)}.
\eeq
This proves
\beqn
\label{eq:Fxdiam}
\abs{F_x(u,v)-F_x(u',v')}
&\le&
4\eps\diam(A)
.
\eeqn
To prove the remaining claim, we argue that for
$\eps<R^*/2$, all $(u,v)\in W(\eps)$
satisfy 
$u,v\in B(x,2\ninf{\evalat{f}{A}}/R^*)$.
Indeed,
assume for
a contradiction that some 
$(u,v) \in W(\e)$ 
violates this assumption,
with
$\rho(x,v)\vee \rho(x,u) > 2\ninf{\evalat{f}{A}}/R^*$. 
Then,
\beq
R(u,v) 
= 
\frac{f(v)-f(u)}{\rho(x,v)+\rho(x,u)}
\le
\frac{\ninf{\evalat{f}{A}}}{\rho(x,v)+\rho(x,u)}
\le
\frac{R^*}{2}
<
R^* - \e,
\eeq
implying that $(u,v) \notin W(\e)$, a contradiction.
Since the diameter of the point pairs in $W(\eps)$
is at most $4\ninf{\evalat{f}{A}}/R^*$,
we can repeat the calculation leading to
(\ref{eq:Fxdiam}) to complete the proof.
\end{proof}

\begin{corollary}
\label{cor:PMSE-def}
For given
$f:\X\to\R$, $x\in\X$ and $A\subseteq\X$, $|A|\ge2$,
if 
the \extname\ existence condition
\beqn
\label{eq:PMSE-exist}
\lip{\evalat{f}{A}}
\vee
(\diam(A)\wedge\ninf{\evalat{f}{A}})
<\infty
\eeqn
is met,
then
(\ref{eq:PMSE-def})
is well-defined.
\end{corollary}

\begin{remark}
\label{rem:PMSE-argmax}
When $R_x(\cdot,\cdot)$ has a unique maximizer $(u^*,v^*)\in A^2$, the definition of $f_A$ simplifies to
\beqn
\label{eq:fA-argmax}
f_A(x) &=& f(u^*)
+
\frac{\rho(u^*,x)}{\rho(u^*,x)+\rho(v^*,x)}
(f(v^*)-f(u^*)).
\eeqn
In light of Corollary~\ref{cor:PMSE-def},
when (\ref{eq:PMSE-exist}) holds,
there is no loss of generality in assuming that for each $x\in\X$,
there is a unique maximizer $(u^*,v^*)=(u^*(x),v^*(x))$.
In particular, (\ref{eq:fA-argmax}) shows that
for finite $A$, one can compute $f_A(x)$ at any given $x$ in time $O(|A|^2)$.
\end{remark}

For the remainder of this section, $|A|\ge2$
and (\ref{eq:PMSE-exist})
are assumed.
It is readily verified that for $x \in A$, we have $f(x) = f_A(x)$;
thus \extname\ is indeed an extension.

\begin{theorem}[Pointwise minimality of the \extname]
\label{thm:PMSE}
For $A \subseteq \X$ and $f:\X\to\R$,
let $f_A$ be the extension of $f$ from $A$ to $\X$.
Then
\beq
\La_{f_A}(x,\X) \le \La_{f}(x,\X),\qquad x\in\X.
\eeq
\end{theorem}

\begin{proof}
We break down the proof into three shorter claims. As argued in Remark~\ref{rem:PMSE-argmax}, 
there is no loss of generality in assuming, for any $x\in\X$, a unique maximizer $(u^*,v^*)=(u^*(x),v^*(x))$ of $R_x(u,v)$ over $A^2$.

\paragraph{Claim I.}
For any $x \in \X\setminus A $, 
\extname\
achieves the minimum local slope 
on
$A$ among all functions that agree with $f_A$ on $A$:

\beq
\La_{f_A}(x,A)
\le
\La_{f}(x,A)
,
\qquad x\in\X\setminus A.
\eeq

We first show that $f_A$ achieves the optimal slope at $x$ with respect to $(u^*,v^*)
$, which define $f_A(x)$ as in (\ref{eq:fA-argmax}). 
It is enough to show that 
$\frac{|f_A(u^*) -f_A(x)|}{\rho(u^*,x)} = \frac{|f_A(v^*) -f_A(x)|}{\rho(v^*,x)}$, 
since any other value of $f_A(x)$, for which the equality does not hold, will result in a larger slope between $x$ and either $u^*$ or $v^*$. 
In light of (\ref{eq:fuFuvfv}), there is no loss of generality in assuming
$f_A(u^*) \leq f_A(x) \leq f_A(v^*)$. Then:
\begin{equation} \label{eq1}
\begin{split}
 \frac{|f_A(u^*) - f_A(x)|}{\rho(u^*,x)} & = 
 \frac{\left[ f(u^*)+\frac{\rho(u^*,x)}{\rho(v^*,x)+\rho(u^*,x)}(f(v^*)-f(u^*))\right]\ - f(u^*)}{\rho(u^*,x)} \\
 & = \frac{f(v^*)-f(u^*)}{\rho(v^*,x)+\rho(u^*,x)} \\
 & = \frac{f(v^*) - \left[ f(v^*)-\frac{\rho(v^*,x)}{\rho(v^*,x)+\rho(u^*,x)}(f(v^*)-f(u^*))\right]}{\rho(v^*,x)} \\
 & = \frac{|f_A(v^*) - f_A(x)|}{\rho(v^*,x)}.
\end{split}
\end{equation} 
Let $\ell_1 :=  \frac{|f_A(u^*) -f_A(x)|}{\rho(u^*,x)} = \frac{f(v^*)-f(u^*)}{\rho(v^*,x)+\rho(u^*,x)}$. 
It remains to show that
$\frac{|f_A(x') -f_A(x)|}{\rho(x',x)} \leq \ell_1$
for all $x' \in A$. 
Assume, 
to the contrary,
the existence of an $x' \in A$ such that 
$\frac{f_A(x') -f_A(x)}{\rho(x',x)} > \ell_1$. 
Then by (\ref{eq:abcd}),
\beq
\frac{f(x')-f(u^*)}{\rho(x',x)+\rho(x,u^*)} = \frac{f_A(x')-f_A(x)+f_A(x)-f_A(u^*)}{\rho(x',x)+\rho(x,u^*)}> \ell_1,
\eeq
which contradicts the 
definition of $(u^*,v^*)$ 
in (\ref{eq:fA-argmax})
as maximizers of $R_x$.
This proves Claim I.

\paragraph{Claim II.}

\beq
\La_{f_A}(x,\X\setminus A) 
&\leq&
\La_{f_A}(x,A)
,
\qquad x\in\X\setminus A.
\eeq
Let us define the {\em slope} operator $S(u,v):=|f_A(u)-f_A(v)|/\rho(u,v)$. 
It suffices to show that 
\beq
S(x,y) \le \La_{f_A}(x,A) \wedge \La_{f_A}(y,A),
\qquad x,y\in\X\setminus A.
\eeq
Let 
$(u^*(x),v^*(x))$ and
$(u^*(y),v^*(y))$ 
be as defined in (\ref{eq:fA-argmax}).
It follows from the proof of Claim I that $S(x,u^*(y)) \vee S(x,v^*(y)) \leq \La_{f_A}(x,A)$. 

Assume for concreteness that
$\La_{f_A}(x,A)\le \La_{f_A}(y,A)$.
As in the proof of Claim I, 
$f_A(u^*(x)) \le f_A(x) \le f_A(v^*(x))$ 
and 
$f_A(u^*(y)) \le f_A(y) \le f_A(v^*(y))$.
Suppose, for a contradiction, that $S(x,y) > \La_{f_A}(x,A)$. 

If $f_A(x) \leq f_A(y)$, 
then $$f_A(v^*(y)) = f_A(x) +
\rho(x,y)S(x,y) +\rho(y,v^*(y))\La_{f_A}(y) > f_A(x)+\rho(x,v^*(y)) \La_{f_A}(x),$$ implying that
$S(x,v^*(y)) > \La_{f_A}(x)$ --- a contradiction. 

If $f_A(x) > f_A(y)$, 
then 
$$f_A(x) = f_A(u^*(y)) + \rho(u^*(y),y)\La_{f_A}(y) + \rho(y,x)S(x,y) > f_A(u^*(y)) + \rho(u^*(y),x)\La_{f_A}(x),$$
implying that
$S(x,u^*(y)) > \La_{f_A}(x)$ --- a contradiction, which proves Claim~II.

\paragraph{Claim III.}
\beq
\La_{f_A}(x,\X) = \La_{f_A}(x, A) \le \La_{f}(x,\X),\qquad x\in A.
\eeq
Assume for a contradiction that for some $y \notin A$ we have 
\beq
\La_{f_A}(x,\X) \geq
S(x,y)=
\frac{\abs{f_A(x) - f_A(y)}}{\rho(x,y)} 
> 
\La_{f_A}(x,A)
.
\eeq
Let $(u^*(y),v^*(y)) \in A^2$ be the maximizer defining $f_A(y)$,
as in Remark~\ref{rem:PMSE-argmax}.
Since $x \in A$, Claim I implies that
$
S(y,u^*(y)) = S(y,v^*(y)) \geq S(x,y)$. 
Then by (\ref{eq:abcd}),
either $f_A(x) \geq f_A(y)$ satisfying
$$\frac{f_A(x)-f_A(u^*(y))}{\rho(x,y)+\rho(u^*(y),y)} 
=
\frac{f_A(x) - f_A(y) + f_A(y) - f_A(u^*(y))}{\rho(x,y)+\rho(u^*(y),y)}
>
\La_{f_A}(x,A)
,
$$
or $f_A(y) > f_A(x)$ satisfying
$$\frac{f_A(v^*(y))-f_A(x)}{\rho(v^*(y),y)+\rho(x,y)} 
=
\frac{f_A(v^*(y)) - f_A(y) + f_A(y) - f_A(x)}{\rho(v^*(y),y)+\rho(x,y)}
>
\La_{f_A}(x,A)
.
$$
Either of these contradicts the maximizer property of $(u^*(y), v^*(y))$,
which proves Claim III.

\paragraph{Putting it together.}
Combining Claims II and III
yields
that the local slope of $x$ with respect to $f_A$
is determined by a point in $A$:
\beq
\La_{f_A}(x,\X) 
= 
\La_{f_A}(x,A),
\qquad x\in\X.
\eeq
Therefore:
\begin{equation} \label{eq2}
\begin{split}
 \La_{f_A}(x,\X) & = 
 \La_{f_A}(x,A) 
 \leq \La_{f}(x,A)
 \leq \La_{f}(x,\X)
\end{split}
\end{equation}
where the first inequality stems from Claim I and the fact that $f$ and $f_A$ agree on $A$.
\end{proof}

\begin{corollary}[Properties of the \extname]
\label{cor:PMSE-prop}
Suppose that $A \subseteq \X$
and $f:\X\to\R$
satisfy (\ref{eq:PMSE-exist}).
Then for any $x \in \X$,
$B \subseteq A$,
and $(u^*, v^*)=(u^*(x), v^*(x))$ as defined in
Remark~\ref{rem:PMSE-argmax},
the following hold:
\begin{enumerate}
    \item[(i)] {Local slope value}:
    $$\La_{f_A}(x,\X) = \frac{\abs{f_A(x) - f_A(u^*)}}{\rho(x,u^*)} = \frac{\abs{f_A(x) - f_A(v^*)}}{\rho(x,v^*)} = \frac{\abs{f_A(v^*) - f_A(u^*)}}{\rho(v^*,x) + \rho(x,u^*)};$$
    \item[(ii)] {Local slope bounds}:
    $$\La_{f_A}(x,\X) \leq \La_{f_A}(u^*,A) \wedge \La_{f_A}(v^*,A) \leq \La_{f_A}(u^*,\X) \wedge \La_{f_A}(v^*,\X);$$
    \item[(iii)] {Lipschitz}: $\lip{f_A} = \lip{\evalat{f}{A}}$;
    \item[(iv)] {Local slope monotonicity}: $\La_{f_{B}}(x,\X) \leq \La_{f_A}(x,\X)$;
    \item[(v)] {Extension sandwich}: $f_A(u^*) \leq f_A(x) \leq f_A(v^*)$.
\end{enumerate}
\end{corollary}
\begin{proof}
All of the claims follow from Theorem~\ref{thm:PMSE} and its proof:
\begin{enumerate}
    \item[(i)] Follows from Claim II and the proof of Claim I.
    \item[(ii)] Follows from (i): $\La_{f_A}(x,\X) = \frac{\abs{f_A(v^*) - f_A(u^*)}}{\rho(v^*,x) + \rho(x,u^*)} \leq \frac{\abs{f_A(v^*) - f_A(u^*)}}{\rho(v^*,u^*)} \leq \La_{f_A}(u^*,A)$.
    \item[(iii)] Follows from (ii): $\La_{f_A}(x,\X) \leq \La_{f_A}(u^*,A) \leq \lip{\evalat{f}{A}}$.
    \item[(iv)] Direct result of Theorem~\ref{thm:PMSE} 
    where $f_A$ assumes the role of $f$.
    \item[(v)] Was proved in the course of proving Claim I.
\end{enumerate}
\end{proof}

%%<AK_TEX_SCRIPT_CODE_DON'T_ALTER_OR_DUPLICATE|"aux-results.tex"
\section{Auxiliary results.}

\begin{lemma}
\label{lem:strong-weak-log}
Suppose 
that $\X$ is a finite set and
$X:\X\to\R_+$ is a random variable with range
$R=X(\X)\subset\R$.
Then
\beq
\E[X] &\le& 2\W[X]{\log\frac1{p_*}}
,
\eeq
where $p_*:=\min\set{
\P(X=r)\neq0: r\in R}$
and 
$\W[X]
:=
\sup_{t>0}t\P(X\ge t)
$
is the ``weak mean'' of $X$.

\end{lemma}
\begin{proof}
Put $r^*=\max R$. Then
\beq
\E[X] &=& \int_0^{r^*}\P(X\ge t)\d t \\
&\le& a + \W[X]\int_a^{r^*}\frac{\d t}{t},
\qquad 0< a\le r^*.
\eeq
The integral evaluates to $\log(r^*/a)$,
and
our choice $a:=p_*r^*$ yields the bound
\beq
\E[X] &\le& p_*r^* + \W[X]\log\frac{1}{p_*}
\\ 
&\le& 
\W[X] + \W[X]\log\frac{1}{p_*}\\
&\le& 2
\W[X]
\log\frac{1}{p_*}
.
\eeq

\end{proof}
\begin{corollary}
\label{cor:s<=wlogn}
Let $(\X,\rho)$ be a finite metric space
endowed with the uniform
distribution $\mu$.
Then, for any $f:\X\to\R$, we have
\beq
\bLas_f(\mu,\X)
&\le&
2{\log(|\X|)}
\bLaw_f(\mu,\X).
\eeq
\end{corollary}

\subsection{Adversarial extension: deferred proofs}

%%<AK_TEX_SCRIPT_CODE_DON'T_ALTER_OR_DUPLICATE|"adv-ext-aux-proofs.tex"
\begin{proof}[Proof of Lemma~\ref{lem:pose-smooth}]
By 
Corollary~\ref{cor:PMSE-prop}(i),
we have that the local slope of $f$ at $x$ is determined by a pair of points $u^*,v^* \in \X_n$:
\beq
\La_f(x) = \frac{f(v^*)-f(u^*)}{\rho(v^*,x)+\rho(x,u^*)}.
\eeq
From (\ref{eq:E}), we have 
$\rho(v^*,x) + \rho(x,u^*) \geq \rho(v^*,u^*) 
\ge 2\diam(E)$,
and hence
$\rho(v^*,x) + \rho(x,u^*) + 2\diam(E) \leq 2(\rho(v^*,x) + \rho(x,u^*))$.
Thus,
\begin{align*}
\La_{f}(x')
& \geq
\frac{f(v^*)-f(u^*)}{\rho(v^*,x')+\rho(x',u^*)}\\
& \geq
\frac{f(v^*)-f(u^*)}{\rho(v^*,x) + \diam(E) + \rho(x,u^*) + \diam(E)} \\
& \geq \frac{f(v^*)-f(u^*)}{2(\rho(v^*,x)+ \rho(x,u^*))} 
= \frac{\La_{f}(x)}{2}
.
\end{align*}
\end{proof}

\begin{proof}[Proof of Lemma~\ref{lem:cell-rec}]
The claims in
(\ref{eq:heavy-big})
and
(\ref{eq:heavy-small})
are standard applications of multiplicative Chernoff bounds
\citep[Theorem 1.1]{DBLP:books/daglib/0025902}.
To prove (\ref{eq:light-big}), observe that
\beq
\P\paren{
\sum_{B\in\Pi_0}\mu(B)\ge 2q
}
&\le&
\P\paren{
\sum_{B\in\Pi_0}(\mu(B)-\mu_n(B))\ge q
}\\
&\le&
\P\paren{
\sum_{B\in\Pi}|\mu(B)-\mu_n(B)|\ge q
}.
\eeq
To bound the latter,
define the random variable
$J_n:=\sum_{B\in\Pi}|\mu(B)-\mu_n(B)|$.
It follows from
\citet[Eqs. (5) and (17)]{Berend2013}
that $\E[J_n]\le\sqrt{m/n}$
and
\beq
\P(J_n\ge q) &\le&
\P(J_n\ge \E[J_n]+(q-\sqrt{m/n}))\\
&\le&
\exp[-n(q-\sqrt{m/n})^2/2],
\qquad nq^2\ge m,
\eeq
which completes the proof.
\end{proof}

\begin{proof}[Proof of Lemma~\ref{lem:max-alg-lip}]
We only prove the first inequality,
since the second one is immediate
from (\ref{eq:blasw}).
It follows directly from the definition
of $\bLaw_y(\mu_n,\X_n)$
that 
\beq
\mu_n(M_y(L))\ge\alpha
\implies
L\le \alpha\inv \bLaw_y(\mu_n,\X_n).
\eeq

The algorithm removes the 
$\floor{\eps n}$ points with the largest $\La_y(x,\X_n)$ from the set $\X_n$ 
and it is a basic fact that
\beq
{\eps n}
\le 2
\floor{\eps n},
\qquad
\eps>0, n\ge\eps\inv.
\eeq
This corresponds to removing a mass of
$\alpha\ge\e/2$ points,
and so
for 
$n \geq {\eps\inv}$,
\beqn
\label{eq:lmax}
\max_{x \in \Xec}
\La_y(x,\Xec) \leq \frac{2\bLaw_y(\mu_n,\X_n)}{\eps}.
\eeqn
Finally, Corollary~\ref{cor:PMSE-prop}(iii,iv)
implies
\beq
\lip{f} = 
\lip{\evalat{f}{V}} 
\leq \lip{\evalat{f}{\Xec}} \leq \frac{2\bLaw_y(\mu_n,\X_n)}{\eps}
.
\eeq
\end{proof}

%%<AK_TEX_SCRIPT_CODE_DON'T_ALTER_OR_DUPLICATE|"generalization.tex"
\section{Generalization}
\label{sec:gen}
{\em Generalization guarantees} refer to claims
bounding the true risk in terms of the empirical risk,
plus confidence and hypothesis complexity terms.
Throughout this paper, we assume that the learner has
a fixed
maximal allowable average slope $L$.
This assumption incurs no loss of generality,
since a standard technique, known as Structural Risk Minimization
(SRM), \citep{DBLP:journals/tit/Shawe-TaylorBWA98},
creates a nested family of function classes with increasing
$L$, allowing the learner to select one
based on the sample, so as to optimize
the underfit-overfit tradeoff.

\subsection{Uniform Glivenko-Cantelli}
We begin by dispensing with some measure-theoretic
technicalities. Our learner constructs a hypothesis
$f:\X\to[0,1]$ via the \extname\ extension, which,
by Corollary~\ref{cor:PMSE-prop}(iii), is a
Lipschitz function.
Thus, operationally, the learner's function class is
\beqn
\label{eq:HL}
\H_L = 
\barLw_L(\X,\rho,\mu)
\cap
\set{f:\X\to\R; \lip{f}<\infty}.
\eeqn
Our assumptions that $\ddim(\X),\diam(\X)<\infty$
imply that $(\X,\rho)$ has compact closure;
this in turn implies a countable
$\F\subset\H_L$ such that every member
of $\H_L$ is a pointwise limit of a sequence in $\F$.
This suffices 
\citep{dudley1999uniform}
to ensure the measurability
of the empirical process
$
\sup_{f\in\H_L}
\paren{
\int_\X f\d\mu
-
\int_\X f\d\mu_n
}
.
$

A standard method for bounding 
this empirical process
is via the Rademacher complexity
(see, e.g., \citet[Theorem 3.1]{mohri-book2012}):
with probability at least $1-\delta$,
\beqn
\label{eq:mohri}
\sup_{f\in\H_L}
\paren{
\int_\X f\d\mu
-
\int_\X f\d\mu_n
}
&\le&
2\radem_n(\H_L|\Xn)+3\sqrt{\frac{\log(2/\delta)}{2n}},
\eeqn
where $\Xn=(X_1,\ldots,X_n)\sim\mu^n$ and
\beq
\radem_n(\H_L|\Xn) :=
\E_{\sigma\sim
  \mathrm{Uniform}(\set{-1,1}^n)}
\sup_{f\in\H_L}\frac1n\sum_{i=1}^n \sigma_if(X_i).
\eeq

Finally, an elementary estimate on Rademacher complexity is in terms of the empirical $\Lp2$ covering numbers (see, e.g., \citet{bartlett-notes}):
\beqn
\label{eq:bartlett}
\radem_n(\H_L|\Xn) \le \inf_{\eps>0}\eps + 
\sqrt{
  \frac{2\log
\covn(
\eps
,
\H_L
,
\Lp2(\mu_n) 
)}{n}}.
\eeqn

Invoking the estimate in Theorem~\ref{thm:emp-cov-num}
with $\alpha=L^{-1/6}$ and combining it with
(\ref{eq:mohri}, \ref{eq:bartlett})
yields that
\beqn
\label{eq:HL-radem}
\sup_{f\in\H_L}
\paren{
\int_\X f\d\mu
-
\int_\X f\d\mu_n
}
\le
\frac{C_\delta \sqrt{L}}
{n^{1/8d}}
+
\frac{C_\delta^{-d/2} \sqrt{2}} {n^{5/16}}
+
3\sqrt{\frac{\log(2/\delta)}{2n}}
\eeqn
holds with probability at least
$1-3\delta$,
where $d=\ddim(\X)$
and $C_\delta$ is a constant depending only on $\delta$ (assuming, to avoid trivialities, that $L,d\ge1$). We have not attempted to optimize
any of the constants.

\subsection{Risk bounds}
\label{sec:risk-bounds}
Recall our learning setup:
$\joint$ is an unknown distribution on $\X\times[0,1]$,
from which the learner receives a labeled sample
$S_n=(X_i,Y_i)_{i\in[n]}\sim\joint^n$.
Based on $S_n$, the learner constructs a
{\em hypothesis} $f:\X\to[0,1]$, to which
we associate the (true) risk
$R(f;\joint):=\E_{(X,Y)\sim\joint}|f(X)-Y|$
as well as the empirical risk $R(f;\joint_n)$.

We discuss regression and classification separately.

\paragraph{Regression.}
The learner selects a hypothesis
$f\in\H_L$ (seeking to minimize $R(f;\joint_n)$, but
this will not be used in our analysis). Then
\beq
\sup_{f\in\H_L}(R(f;\joint_n)-R(f;\joint))
&\le&
2\radem(\G|(X,Y)_{[n]})
+3\sqrt{\frac{\log(2/\delta)}{2n}}
\eeq
holds with probability at least $1-\delta$,
where $\G$ is the {\em loss class} consisting of
\beq
\G=\set{ 
g:(x,y)\mapsto\abs{f(x)-y};
x\in\X, y\in[0,1], f\in\H_L
}\subseteq [0,1]^{\X\times[0,1]}.
\eeq
To bound the Rademacher complexity
of $\G$, we notice that each $g\in\G$
is of the form $g=\phi\circ h$,
where $\phi(\cdot)=\abs{\cdot}$
and $h:(x,y)\mapsto f(x)-y$.
Since $\phi:\R\to\R$ is $1$-Lipschitz,
Talagrand's contraction principle 
\citep[Corollary 3.17]{LedouxTal91}
implies that
\beq
\radem(\G|(X,Y)_{[n]}) &\le&
\E_{\sigma}
\sup_{f\in\H_L}
\frac1n\sum_{i=1}^n \sigma_i(f(X_i)-Y_i)\\
&=&
\E_{\sigma}
\sup_{f\in\H_L}
\frac1n\sum_{i=1}^n \sigma_if(X_i)
=
\radem(\H_L|\Xn).
\eeq
Combining this with (\ref{eq:HL-radem})
yields
\beqn
\label{eq:regbounds}
\sup_{f\in\H_L}(R(f;\joint_n)-R(f;\joint))
\le
\frac{C_\delta \sqrt{L}}
{n^{1/8d}}
+
\frac{C_\delta^{-d/2} \sqrt{2}} {n^{5/16}}
+
3\sqrt{\frac{\log(2/\delta)}{2n}}
\eeqn
with probability at least $1-3\delta$.

\paragraph{Classification.}
For classification,
the learning setup is asymmetric with respect
to training and prediction (see the discussion at the beginning of
Section~\ref{sec:adv-ext-class}). The distribution
$\joint$ is over $\X\times\set{0,1}$, and again,
$S=(X_i,Y_i)_{i\in[n]}\sim\joint^n$ is presented to the learner. The latter produces a hypothesis
$f:\X\to[0,1]$, to which we associate the
{\em sample error},
\beqn
\label{eq:serr}
\serr(f) = \frac1n\sum_{i=1}^n \pred{f(X_i)\neq Y_i},
\eeqn
and the {\em generalization error},
\beqn
\label{eq:gerr}
\gerr(f) = \P_{(X,Y)\sim\joint}( 
\Int{f(X)}
\neq Y ).
\eeqn
Notice the asymmetry between $\serr$ and $\gerr$:
a hypothesis is penalized for every sample point
it fails to label correctly, but is only required
to ``be closer to the correct label'' at test time.

Given these definitions, a standard argument (via the {\em margin function} and Talagrand's contraction, see
\citet[Theorem 4.4]{mohri-book2012}),
yields
\beqn
\label{eq:class-gen}
\sup_{f\in\H_L}(\gerr(f)-\serr(f))
&\le&
\frac{C_\delta \sqrt{L}}
{n^{1/8d}}
+
\frac{C_\delta^{-d/2} \sqrt{2}} {n^{5/16}}
+
3\sqrt{\frac{\log(2/\delta)}{2n}}
\eeqn
with probability at least $1-3\delta$.

%%<AK_TEX_SCRIPT_CODE_DON'T_ALTER_OR_DUPLICATE|"adv-ext-regr-weak-mean.tex"
\section{Adversarial extension for regression, weak mean}
\label{sec:adv-ext-regr-weak-mean}
\newcommand{\lmin}{\ell_{\min}}
\newcommand{\lmax}{\ell_{\max}}
In this section, we prove the weak-mean counterparts (a.ii), (b.ii)
of the adversarial extension game for regression, defined in
Section~\ref{sec:adv-ext}. These results are not used in the paper and are included
for completeness and independent interest.
The following notation will be used
throughout:
\beq
\lmin &:=& \min_{x\in\X_n}\La_f(x,\X_n),\\
\lmax &:=& \max_{x\in\X_n}\La_f(x,\X_n)
.
\eeq
The function $f:\X\to[0,1]$
will refer exclusively to the one constructed by the 
``adversarial extension'' algorithm
in the proof of Lemma~\ref{lem:adv-ext-strong}.
We proceed to make some observations
regarding~$f$.

\begin{lemma}[$f$ is defect-free]
For all $\ell>0$,
the function $f$
is
$(\ell\e/2,\ell,1)$-defect-free,
as defined in
Section~\ref{sec:defects-defs}.
\end{lemma}
\begin{proof}
By 
construction,
$f$ is the \extname~of an $\e$-net $V$. 
Let 
$u^*(x)$
and
$v^*(x)$
be as in Remark~\ref{rem:PMSE-argmax}.
Suppose that $\La_f{(x,\X)} \geq \ell$. Then
\beq
&&
\max\{|f(x)-f(u^*(x))|,|f(x)-f(v^*(x))|\}\\
&\geq&
\ell\cdot
\max\{\rho(x,u^*(x)),\rho(x,v^*(x))\}
\geq
\ell\cdot\frac{\e}{2}
=
\ell\e/2
.
\eeq
\end{proof}

\begin{lemma}[Combinatorial structure]
\label{lem:comb-weak}
For any $t>0$
and $0<q<1$,
we have
\beqn
\label{eq:weak-comb-struct}
\mu(M_f(t)) \leq 2\mu_n(M_f(t/4)) + \frac{8m}{n^{1/3}}\log(m/\delta)
,
\eeqn
with probability at least $1-2\delta$,
where
$m \leq (8/\e)^{d}$.
\end{lemma}

\begin{lemma}[Bounded local slope ratio]
\label{lem:slope-rat}
\beq
\max_{x \neq x' \in \X}{\frac{\La_f(x,\X)}{\La_f(x',\X)}}
&\leq&
\frac{
2\diam(\X)}
{\e}
.
\eeq
\end{lemma}

Proofs of Lemmas~\ref{lem:comb-weak} and \ref{lem:slope-rat} are deferred to the end of this section.

\begin{lemma}[$f$ is close to $y$ in weak mean]
\label{lem:a.ii}
The ``adversarial extension'' function $f$
satisfies (a.ii) for $0<\e<1$.
\end{lemma}
\begin{proof}
We begin with the same decomposition
as in (\ref{eq:a.i-decomp}):
\beq
\nrm{f-y}_{\Lp{1}(\mu_n)}
&=&
\frac{1}{n}\sum_{x\in
\Xe
\setminus\enet
}|f(x)-y(x)|
+
\frac{1}{n}\sum_{x\in\Xec
\setminus\enet
}|f(x)-y(x)|
\eeq
and, as above, bound the first term by $\e$
and the second term as in 
(\ref{eq:a.i-2}):
\beq
\frac{1}{n}\sum_{x\in\Xec
\setminus\enet
}|f(x)-y(x)|
&\le&
\frac{2\e}{n}
\sum_{x\in\Xec\setminus\enet}
\La_y(x,\X_n)
\\
&=&
2\e\frac{|\Xec\setminus\enet|}n
\int_{\Xec\setminus\enet}
\La_y(x,\X_n)
\d\bar\mu(x)
,
\eeq
where $\bar\mu$ is
is the uniform measure on
$\Xec\setminus\enet$,
given by
$\bar\mu(x)=
n|\Xec\setminus\enet|\inv\mu_n(x)$.
Now
\beq
\frac{|\Xec\setminus\enet|}n
\int_{\Xec\setminus\enet}
\La_y(x,\X_n)
\d\bar\mu(x)
&=&
\frac{|\Xec\setminus\enet|}n
\int_0^\infty
\bar\mu(\set{x\in \Xec\setminus\enet:\La_y(x,\X_n)>t})\d t\\
&=&
\int_0^\infty
\mu_n(\set{x\in \Xec\setminus\enet:\La_y(x,\X_n)>t})\d t\\
&=&
\int_0^{\beta}
\mu_n(\set{x\in \Xec\setminus\enet:\La_y(x,\X_n)>t})\d t\\
&\le&
\alpha+
\bLaw_y(\mu_n,\X_n)
\int_\alpha^\beta
\frac{\d t}t,
\eeq
where $\alpha>0$ is arbitrary
and $\beta=
\frac{2\bLaw_y(\mu_n,\X_n)}{\eps}
$, in light of (\ref{eq:lmax}).
The integral evaluates to $\log(\beta/\alpha)$
and our choice $\alpha:=\bLaw_y(\mu_n,\X_n)$
yields the estimate
\beq
\frac{1}{n}\sum_{x\in\Xec
\setminus\enet
}|f(x)-y(x)|
&\le&
2\e(\alpha+\alpha\log\frac2\e)
=\tilde O(\e)\bLaw_y(\mu_n,\X_n)
.
\eeq
\end{proof}

\begin{lemma}[Satisfying (b.ii)]
\label{lem:b.ii}
For $0<\e,\delta<1$,
the adversarial extension function
$f$ satisfies
\beq
\bLaw_f(\mu,\X)
&\le&
16\bLaw_f(\mu_n,\X_n)
+
64
\bLaw_y(\mu_n,\X_n)
\frac{m
}{\e n^{1/3}}\log\left(\frac{m\log
(2/\e)}{\delta}\right)
\eeq
with probability at least $1-2\delta$,
where
$m\le(8/\eps)^{d}$.
\end{lemma}
\begin{proof}
Fix a 
$\delta>0$
and put $\delta_{ij}:=\delta2^{-i-j}$;
then $
\sum_{i=1}^\infty
\sum_{j=1}^\infty
\delta_{ij}
=\delta$.
Define also
$\tau_{ij}=2^{i-j}$
for $i,j\in\set{1,2,\ldots}$.
Invoking Lemma~\ref{lem:comb-weak} with the union bound,
we have:
\beqn
\label{eq:tauij}
\P\paren{
\exists \tau_{ij}
:
\mu(M_f(\tau_{ij})) 
> 
2\mu_n(M_f(\tau_{ij}/4)) + \frac{8m}{n^{1/3}}\log\frac{m2^{i+j}}{\delta}
}
\le2\delta.
\eeqn
For $t<\lmin$, we have
\beqn
\label{eq:tlmin}
t\mu(M_f(t))\le t=t\mu_n(M_f(t))
\eeqn
and for $t>\lmax$, 
recalling from Lemma~\ref{lem:max-alg-lip}
that $\lip{f}
=
\lmax
\le
2\e\inv \bLaw_y(\mu_n,\X_n)
$,
we have
\beqn
\label{eq:tlmax}
t\mu(M_f(t))=0=t\mu_n(M_f(t)).
\eeqn
For 
any other
$t>0$, define $\tau(t)$ to be the largest
$\tau_{ij}\le t$.
Then, for $t\in[\lmin,\lmax]$,
\beq
t\mu(M_f(t)) &\le&
2\tau(t)\mu(M_f(\tau(t))) \\
&\le_p& 
4\tau(t)\mu_n(M_f(\tau(t)/4)) +
16\tau(t)\frac{m}{n^{1/3}}\log\frac{m\log(\lmax/\lmin)}{\delta}\\
&\le&
16(\tau(t)/4)\mu_n(M_f(\tau(t)/4)) +
32\lmax\frac{m}{n^{1/3}}\log\frac{m\log(\lmax/\lmin)}{\delta}\\
&\le&
16\bLaw_f(\mu_n,\X_n)
+
64
\e\inv \bLaw_y(\mu_n,\X_n)
\frac{m
}{n^{1/3}}\log\left(\frac{m\log
(2/\e)}{\delta}\right),
\eeq
where 
$\le_p$ holds with probability at least
$1-\delta$
and
Lemma~\ref{lem:slope-rat} was invoked
in the last inequality.
The claim follows.

\end{proof}

\paragraph{Defered proofs.}

\begin{proof}[Proof of Lemma~\ref{lem:comb-weak}]
Let $\Pi$ be the partition defined in Lemma~\ref{lem:M_f(ell)} and let $\Pi = \Pi_0 \cup \Pi_1$ be the dichotomy of $\Pi$ into light and heavy cells as in the proof of Lemma~\ref{lem:adv-ext-strong}. Put 
$U_f = U_{0} \cup U_{1}$, where:
\beq
U_{0} &=& \set{B \in \Pi_0: 
B\cap M_f(t)\neq\emptyset 
},
\\ 
U_{1} &=& \set{B \in \Pi_1 :
B\cap M_f(t)\neq\emptyset 
}.
\eeq
Then
\beq
\mu(M_f(t)) &=& 
\sum_{B \in \Pi_0}{\mu(B\cap M_f(t))} + \sum_{B \in \Pi_1}{\mu(B\cap M_f(t))}
\\
&\leq&
\sum_{B \in U_{0}}{\mu(B)} + 
\sum_{B \in U_{1}}{\mu(B)}
\\
&\le_p&
(2\mu_n(\cup U_{0}) + 2q)
+ 2\sum_{B \in U_{1}}{\mu_n(B)}
\\
&\leq&
2\mu_n(\cup U_{0})
+
2\mu_n(\cup U_{1})
+ 2q
\\
&=&
2\mu_n(U_f)
+ 2q
\\
&\leq&
2\mu_n(M_f(t/4)
+ 2q,
\eeq
where 
$\le_p$
follows from
(\ref{eq:heavy-big},
\ref{eq:light-big}) 
and the final inequality
from
Lemma~\ref{lem:M_f(ell)}.
Hence the bound holds with probability at least 
$
1-
\left[
m
\exp
\left(
-\frac{nq}{4m}
\right)
+
\exp
\left(
-\frac{m+nq^2}{2}+q\sqrt{mn}
\right)
\right]
$.
Choosing $q = (4m/n^{1/3})\log(m/\delta)$, we have that (\ref{eq:weak-comb-struct}) holds with probability at least
\beq
&&
1-
m\left(\frac{\delta}{m}\right)^{n^{2/3}}
-
\exp
\left(
-\frac{m+n(4m/n^{1/3})^2\log(m/\delta)^2}{2}+(4m/n^{1/3})\log(m/\delta)\sqrt{mn}
\right)\\
&\ge&
1 - \delta -
\exp(-m/2)
\cdot
\left(\frac{\delta}{m}\right)^
{
8m^2\log{(m/\delta)}n^{1/3}
-
4m^{3/2}\log{(m/\delta)}n^{1/6}
}
\\
&=&
1 - \delta -
\exp(-m/2)
\cdot
\left(\frac{\delta}{m}\right)^
{
\log(m/\delta)n^{1/6}
(
8m^2n^{1/6}
-
4m^{3/2}
)
}
\\
&\ge&
1-2\delta
.
\eeq
\end{proof}

\begin{proof}[Proof of Lemma~\ref{lem:slope-rat}]
By 
Corollary~\ref{cor:PMSE-prop}(iii), we may assume that for some $u,v \in V$,
\beq
\lip{f} = \frac{|f(u)-f(v)|}{\rho(u,v)}.
\eeq
Since $\rho(u,v) \geq \e$, 
we have,
for any $x \in \X$,
\beq
\La_f{(x,\X)}
&\geq&
\max{\bigg\{\frac{|f(x)-f(u)|}{\rho(x,u)} , \frac{|f(x)-f(v)|}{\rho(x,v)}\bigg\}}
\\
&\geq&
\max{\bigg\{\frac{|f(x)-f(u)|}{\diam(\X)} , \frac{|f(x)-f(v)|}{\diam{(\X)}}\bigg\}}
\\
&\geq&
\frac{|f(u)-f(v)|}{2\diam{(\X)}}
\\
&=&
\lip{f} \cdot \frac{\rho(u,v)}{2\diam{(\X)}}
\geq
\lip{f} \cdot \frac{\e}{2\diam{(\X)}}
.
\eeq
\end{proof}

%%<AK_TEX_SCRIPT_CODE_DON'T_ALTER_OR_DUPLICATE|"illustr-ex.tex"
\section{Illustrative examples and discussion}
\label{sec:disc}

\paragraph{Savings of average over worst-case.}

For $\gamma\in(0,1/2)$, consider the metric space
$\X=[0,1/2-\gamma]\cup[1/2+\gamma,1]$
equipped with the
standard metric $\rho(x,x')=\abs{x-x'}$,
the uniform distribution $\mu$,
and $f:\X\to\R$ given by
the step function $f(x)=\pred{x>1/2}$.
Then $\lip{f}=1/(2\gamma)$ and
\beqn
\label{eq:illustr1}
\bLas_f(\X,\rho,\mu)
=
\frac1{1-2\gamma}
\int_{[0,1/2-\gamma]\cup[1/2+\gamma,1]}\frac{1}{\abs{x-1/2}+\gamma}\d x
=
\frac2{1-2\gamma}
\log\left(\frac{1+2\gamma}{4\gamma}\right).
\eeqn
For small $\gamma$, we have
$\lip{f}=\Theta(\gamma\inv)$
and
$\bLas_f=\Theta(\log\gamma\inv)$,
so even the cruder strong average smoothness measure provides
an exponential savings over the worst-case one.
This example has natural higher-dimensional analogues
(i.e., $\X=[0,1]^{d-1}\times([0,1/2-\gamma]\cup[1/2+\gamma,1])$), where
the phenomenon persists.

An analogous behavior is exhibited by the ``margin loss'' function
$f(x)=\max\set{1,\min\set{0,1-x/\gamma}}$
defined on $([0,1],\abs{\cdot},\mu)
$, where $\mu$ is the uniform distribution on $[0,1]$.
In this case, $\lip{f}=1/\gamma$, while
\beqn
\label{eq:illustr2}
\bLas_f([0,1],\abs{\cdot},\mu)
=
\gamma\cdot\frac1\gamma
+
(1-\gamma)\int_\gamma^1\frac{1}{x}\d x
=1+(1-\gamma)\log\frac{1}{\gamma}
=\Theta\paren{\log\frac1\gamma}
\eeqn
--- again, an exponential savings.

For a more dramatic gap between the two measures,
consider the family of functions $f_p(x)=x^p$ for $p\in(0,1)$,
on $\X=[0,1]$ with the uniform distribution $\mu$ and the standard metric $\rho$.
These all have $\lip{f_p}=\infty$, while
\beqn
\label{eq:illustr3}
\bLas_f([0,1],\abs{\cdot},\mu)
=
\int_0^1x^{p-1}\d x=\frac1p.
\eeqn

Consider now the case of the step function
$f:[0,1]\to[0,1]$ given by $f(x)=\pred{x>0}$.
Taking $\X$ and $\mu$ as above, we have
$\lip{f}=\bLas_f(\mu,\X)=\infty$,
so here the strong mean offers no
advantage over the worst-case. However,
since $\mu(M_f(t))=1/t$, we have
that $\bLaw_f(\mu,\X)=1$.
More generally, in an ongoing work with, A. Elperin,
we have shown that $BV[0,1]$, the class of all bounded-variation
functions on $[0,1]$, satisfies $BV[0,1]\subset \barLw([0,1],\rho,\mu)$,
and the containment is strict.

\paragraph{Uniform Glivenko-Cantelli.}
Take $\X=\set{1,2,\ldots}$ with any probability measure
$\mu$
and metric $\rho(x,x')=|x-x'|$.
Consider the function class $\F=[0,1]^\X$.
It is well-known that $\F$ is
UGC with respect to $\mu$; this follows, for example,
from missing mass arguments
(see, e.g., \citep[Theorem 2]{ENK-20aistats}).
We note in passing that under a fixed distribution, UGC 
is a strictly stronger property than
learnability
\citep{Benedek1991377}. Let us specialize
our general techniques to this toy setting.

It is easily seen that $\ddim(\X)=1$, but we must address
the technical issue that $\diam(\X)=\infty$. A cursory glance at the proof of
Theorem~\ref{thm:emp-cov-num} shows that in fact only $\diam(\X_n)<\infty$ is needed.
In fact, even further savings is possible: we can relabel the elements of $\X_n$ so that
\beq
\diam(\X_n)=\nrm{\mu_n}_0\equiv |\supp(\mu_n)|=: |\set{x\in\X:\mu_n(x)>0}|.
\eeq
Renormalizing the empirical diameter to $1$ by shrinking the distances by $\nrm{\mu_n}_0$,
we have that
$\lip{\evalat{f}{\X_n}}\le\nrm{\mu_n}_0$ for all $f\in\F$.
To this we may apply Lemma~\ref{lem:lipcov}, obtaining
a $t$-covering number bound (under $\ell_\infty$) of
order 
$O(\frac{\nrm{\mu_n}_0}t\log\frac1t)$.
Then the Rademacher bound in (\ref{eq:bartlett}) yields a rate of $O((\nrm{\mu_n}_0/n)^{1/3}))$,
although Dudley's chaining integral 
\citep[Theorem 8.1.3]{MR3837109}
yields the sharper estimate $O((\nrm{\mu_n}_0/n)^{1/2}))$.
A more careful analysis \citep{CKW20} shows that 
essentially the optimal rate is
$O((\nrm{\mu_n}_{1/2}/n)^{1/2}))$.
The estimate in (\ref{eq:HL-radem}) loses out considerably, due to the additive error in the covering number bound in Theorem~\ref{thm:emp-cov-num}; however, it does suffice to conclude that $\F$ is UGC.
Note
that
our
techniques 
establish
finite
empirical $\Lp2$
covering numbers for
this class --- unlike, say,
the covering number estimate of
\citet{MR1965359}, which
requires bounded fat-shattering
dimension.

\paragraph{Comparison between \extname\ and AMLE.}
At first glance, \extname\ might appear similar
to the Absolutely Minimal Lipschitz Extension (AMLE)
\citep{MR1884349,MR2449057}.
It was already observed by \citet{MR2431047}
that the two are distinct.
Note first that AMLE requires a length space 
in order
to be well-defined,
while \extname\ of $f$ from $A\subset\X$ to $\X$
is well-defined as long as either $\diam(A)<\infty$
or $\ninf{\evalat{f}{A}}<\infty$.

A visual comparison of the behaviors of AMLE and \extname\ 
on $\X=[0,1]$ is instructive. We evaluate the step function
$f(x)=\pred{x>1/2}$ at $10$ uniformly spaced ``anchor'' points on $[0,1]$.
The AMLE is just the linear interpolation, illustrated by
the piecewise-linear (blue) curve in Figure~\ref{fig:step-func}.
The behavior of \extname\ is more interesting: at first, the sawtooth (red)
shape looks somewhat odd. Recall, however, that the \extname\ is minimizing
the {\em local} slope at each point, which is affected by the values
at all of the anchor points. Thus, each bottom spike reflects the tension
between the $0$-value at the nearby anchor points and the $1$-value
at the farther anchor points (and similarly 
for the top spikes). The two curves coincide
at the line segment representing the steep rise
between the $5$th and $6$th anchor points.

\begin{figure}[ht]
    \centering
    \includegraphics[width=0.5\textwidth]{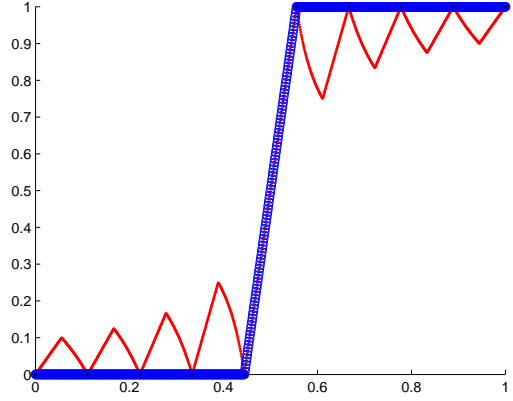}
    \caption{Comparing \extname\ to AMLE on the step function.}
    \label{fig:step-func}
\end{figure}

\paragraph{Alternative notions of average smoothness.}
One might consider a natural alternative definition of average smoothness,
less stringent than our $\bLas$
and
our $\bLaw$:
\beq
\bLalt_f(\mu,\X)
=\int_\X\int_\X \frac{|f(x)-f(y)|}{\rho(x,y)}\d\mu(x)\d\mu(y).
\eeq
We argue that $\bLalt_f(\mu,\X)$ fails as an
average measure of smoothness for bounding
empirical covering numbers and obtaining
generalization guarantees.
Indeed, consider $\X=[0,1]$ endowed with
the standard metric $\rho(x,x')=|x-x'|$
and uniform distribution $\mu$.
Let $\F$ be the collection of all $f:\X\to\set{0,1}$
with finite support. It is well-known
that $\F$ is not UGC under $\mu$, and has
typical empirical covering numbers exponential
in sample size. However, since $|f(x)-f(y)|$
vanishes
$\mu^2$-almost-everywhere on $[0,1]^2$, we have
that $\bLalt_f(\mu,\X)=0$ for all $f\in\F$.
As a consistency check, note that
no uniform bound over all $f\in\F$ is possible
for either
 $\bLas_f$
or $\bLaw_f$.

%%<AK_TEX_SCRIPT_CODE_DON'T_ALTER_OR_DUPLICATE|"multiplicative_chernoff.tex"
\section{Chernoff-type bound}

\begin{theorem}
\label{thm:mult-chernoff}
For
$X\sim\mathrm{Binomial}(n,p)$
and
$p=p(n)\ge 2\log(n)/n$,
$q=q(n)\le
p(n)/\log(n)
$,
\beqn
\label{eq:chernoffpq}
\P(X/n\le q)
&\le& \paren{\frac{e\log n}{n}}^2,
\qquad n\ge3.
\eeqn
\end{theorem}
\begin{proof}

For all $0<q<p<1$ and 
$X\sim\mathrm{Binomial}(n,p)$,
\citep[page 4]{dubhashi98}
\beq
\P(X/n\le q)
&\le&
\paren{
\paren{\frac{p}{q}}^{q}
\paren{\frac{1-p}{1-q}}^{1-q}
}^n
.
\eeq

We first consider the case where 
$p(n)= 2\log(n)/n$
and $q(n)=p(n)/\log(n)=2/n$.
In this case,
\beq
\P(X/n\le q)\cdot\paren{\frac{n}{\log n}}^2
&\le&
\paren{
\paren{\frac{p}{q}}^{q}
\paren{\frac{1-p}{1-q}}^{1-q}
}^n
\cdot\paren{\frac{n}{\log n}}^2\\
&=&
(\log n)^2\cdot\paren{\frac{n-2\log n}{n-2}}^{n-2}
\cdot\paren{\frac{n}{\log n}}^2\\
&=&
n^2\paren{\frac{n-2\log n}{n-2}}^{n-2}=:a(n)\\
&\ntoinf& e^2.
\eeq
Furthermore, the sequence
$a(n)
$
is monotonically increasing for $n\ge 12$,
and it is easily verified that $a(n)\le e^2$ for all $n\ge 3$.
This proves (\ref{eq:chernoffpq}) for 
$p(n)= 2\log(n)/n$, $q(n)=2/n$.

To prove the full claim, consider the function
\beq
f(p,k) = k^{p/k}\paren{\frac{1-p}{1-p/k}}^{1-p/k},
\qquad
p\in[0,1], k\in[1,\infty).
\eeq
We claim that (i) $f$ is monotonically decreasing
in each argument and (ii) this suffices to 
establish (\ref{eq:chernoffpq}) in its full generality.
To see how monotonicity implies the full claim,
note that
$\P(X/n\le p/k)\le f(p,k)^n$,
the latter being maximized by the smallest
feasible values of $p$ and $k$.
The conditions of the Theorem
constrain these at $p\ge 2\log(n)/n$
and $k\ge\log n \ge\log 3>1$, which reduces the
problem to the case analyzed above.

To prove monotonicity in $p$, compute
\beq
\ddel{}{p}\log f(p,k)
=\frac{
1 - k + (1 - p)\log k  - (1 - p)\log[(1 - p)/(1 - p/k)]
}{
k(1-p)
}.
\eeq
Since the denominator is clearly positive, it suffices to prove
that the numerator is negative.
Now $1 - k + (1 - p)\log k\le 1 - k + \log k<0$
for $k>1$, and $(1 - p)/(1 - p/k)>1$,
so the contribution of the remaining term is negative as well.

To prove monotonicity in $k$, compute
\beq
\ddel{}{k}\log f(p,k)
=
\frac{
p [
-\log k + \log( (1 - p)/(1 - p/k))
]
}{
k^2
},
\eeq
whose negativity is equivalent to
$k> (1 - p)/(1 - p/k)$, 
the latter a consequence of $k>1$.

\end{proof}

\end{document}